\documentclass{amsart}

\usepackage{amssymb}
\usepackage{amsthm}

\newtheorem{thm}{Theorem}[section]
\newtheorem{cor}[thm]{Corollary}
\newtheorem{lem}[thm]{Lemma}
\newtheorem{prop}[thm]{Proposition}
\newtheorem{defn}[thm]{Definition}
\newtheorem{rem}[thm]{Remark}

\usepackage{mathrsfs}
\usepackage{amsmath}
\usepackage{amsfonts}
\usepackage{amssymb}
\usepackage{dsfont}
\usepackage{xcolor}
\usepackage{hyperref}
\usepackage{csquotes}

\def\di{\displaystyle}

\def\l{\left}

\def\ae{\alpha}

\def\r{\right}

\def\R{\mathbb{R}}
\def\N{\mathbb{N}}

\def\H{L^{2}(D)}
\def\V{H^{1}(D)}

\def\h{\Delta t}

\def\tk{t_k}
\def\tkp{t_{k+1}}

\def\tn{t_n}
\def\tnp{t_{n+1}}
\def\tnm{t_{n-1}}

\def\unp{\chi_{n+1}}

\def\bnp{B_{n+1}}

\def\Fn{\mathcal{F}_{\tn}}
\def\Fnp{\mathcal{F}_{\tnp}}

\newcommand{\wstarconverge}{\stackrel{*}{\rightharpoonup}}

\begin{document}

\title[Well-posedness for a system of stochastic PDEs]{Well-posedness for the coupling of a random heat equation with a multiplicative stochastic Barenblatt equation}

\author[C. Bauzet]{Caroline Bauzet}
\address{Caroline Bauzet \\ Aix-Marseille Universit\'e, CNRS, Centrale Marseille,
Laboratoire de M\'ecanique et d'Acoustique,
4 impasse Nikola Tesla, 13013 Marseille, France}
\email{bauzet@lma.cnrs-mrs.fr}

\author[F. Lebon]{Fr\'ed\'eric Lebon}
\address{Fr\'ed\'eric Lebon \\ Aix-Marseille Universit\'e, CNRS, Centrale Marseille,
Laboratoire de M\'ecanique et d'Acoustique,
4 impasse Nikola Tesla,
13013 Marseille,
France}
\email{lebon@lma.cnrs-mrs.fr}

\author[A.A. Maitlo]{Asghar Ali Maitlo}
\address{Asghar Ali Maitlo \\ Aix-Marseille Universit\'e, CNRS, Centrale Marseille,
Laboratoire de M\'ecanique et d'Acoustique,
4 impasse Nikola Tesla, 13013 Marseille,
France}
\email{maitlo@lma.cnrs-mrs.fr}

\author[A. Zimmermann]{Aleksandra Zimmermann}
\address{Aleksandra Zimmermann \\ Universit\"at Duisburg-Essen, Fakult\"at f\"ur Mathematik, Thea-Leymann-Str. 9, 45127 Essen, Germany}
\email{aleksandra.zimmermann@uni-due.de}

\begin{abstract}
In this contribution, a stochastic nonlinear evolution system under Neumann boundary conditions is investigated. Precisely, we are interested in finding an existence and uniqueness result for a random heat equation coupled with a Barenblatt's type equation with a multiplicative stochastic force in the sense of It\^{o}. In a first step we establish well-posedness in the case of an additive noise through a semi-implicit time discretization of the system. In a second step, the derivation of continuous dependence estimates of the solution with respect to the data allows us to show the desired existence and uniqueness result for the multiplicative case.
\end{abstract}


\maketitle

\section{Introduction}
We consider the  following system, coupling a heat equation with Barenblatt's one, perturbed firstly by an additive noise:
\begin{equation}\label{eqc1}
\left\{\begin{array}{rcl}
\di \partial_t\vartheta+\partial_t\Big(\chi-\int_0^.hdw\Big)-\Delta\vartheta&=&0 \text{ in }(0,T)\times D\times\Omega,\\
\di \tilde{\ae}\left(\partial_t (\chi-\int_0^. hdw)\right)-\Delta \chi &=&\vartheta \text{ in }(0,T)\times D\times\Omega, \\
\nabla \chi.\mathbf{n}=\nabla\vartheta.\mathbf{n}&=&0 \text{ on } (0,T)\times\partial D\times \Omega,\\
\chi(0,.)=\chi_{0}&\text{and}&\vartheta(0,.)=\vartheta_0,
\end{array}\right.
\end{equation}
and secondly by a multiplicative one:

\begin{equation}\label{eqc123}
\left\{\begin{array}{rcl}
\di \partial_t\vartheta+\partial_t\Big(\chi-\int_0^. \mathscr{H}(\chi)dw\Big)-\Delta\vartheta&=&0 \text{ in }(0,T)\times D\times\Omega,\\
\di \tilde{\ae}\left(\partial_t (\chi-\int_0^. \mathscr{H}(\chi)dw)\right) -\Delta \chi &=&\vartheta\text{ in }(0,T)\times D\times\Omega, \\
\nabla \chi.\mathbf{n}=\nabla\vartheta.\mathbf{n}&=&0 \text{ on } (0,T)\times\partial D\times \Omega,\\
\chi(0,.)=\chi_{0}&\text{and}&\vartheta(0,.)=\vartheta_0,
\end{array}\right.
\end{equation}
where $T>0$, $D$ denotes a smooth and bounded domain of $\R^d$ with $d\geqslant 1$, $\mathbf{ n}$ is the outward normal vector to the boundary $\partial D$, $\chi_0$ and  $\vartheta_0$ are given initial conditions. We consider a standard adapted one-dimensional continuous Brownian motion $$w=\{w_{t}, \mathcal{F}_{t}, 0\leq t\leq T\}$$ defined on a complete probability space $(\Omega,\mathcal{F},P)$ with a countably generated $\sigma$-field denoted $\mathcal{F}$ and a filtration $(\mathcal{F}_{t})_{t\geqslant 0}$ satisfying usual conditions (see \cite{DaPratoZabczyk}, \cite{PrevotRockner} for further informations on stochastic calculus). Let us precise that the additive and multiplicative stochastic integrals $\int_0^. hdw$ and $\int_0^. \mathscr{H}(\chi)dw$  are considered in the sense of It\^o.
\medskip\\
\noindent We assume the following:
\begin{enumerate}
\item[$H_1$:] $h \in \mathcal{N}^2_w(0,T,H^{1}(D))$\footnotemark[1].
\footnotetext[1]{For a given separable Hilbert space $X$, we denote by $\mathcal{N}^2_w(0,T,X)$ the space of predictable $X$-valued processes endowed with the norm  $||\phi||^2_{\mathcal{N}^2_w(0,T,X)}=E\left[\int_0^T||\phi||^2_Xdt\right]$ (see Da Prato-Zabczyk \cite{DaPratoZabczyk} p.94).}
\item[$H_2$:] $\tilde{\ae}=I_d+\alpha$ where $I_d:\R\rightarrow \R$ is the identity function and  $\ae\, :\,\R\to\R$ is a Lipschitz continuous, coercive and non-decreasing function such that $\alpha(0)=0$.
\item[$H_3$:] $\chi_0,\vartheta_0\in H^1(D).$
\item[$H_4$:] $\mathscr{H}:H^1(D)\to H^1(D)$ is a Lipschitz continuous mapping with Lipschitz constant $C_{\mathscr{H}}>0$.
\end{enumerate}

\subsection{State of the art}
\noindent In the deterministic case, i.e., when $h=\mathscr{H}=0$, one application of such nonlinear evolution system is the description of phase transition phenomena, including irreversible phase changes (for instance, solidification of glue, cooking an egg,...) see \cite{BonfantiFremondLuterotti} for further details.\\
Let us mention that Barenblatt's type equations, (namely $f(\partial_t \chi)-\Delta\chi=0$ where $f$ is a non-decreasing function), were initially studied by {\sc G.I. Barenblatt} for the theory of fluids in elasto-plastic porous medium \cite{Barenblatt}, under the assumption that the porous medium is irreversibly deformable. After that, intensive studies have been carried out on this type of equations, see, e.g., \cite{HulshofVazquez,Ig,KaminPeletierVazquez} for more details. Moreover, this type of equations appear in various applications: irreversible phase change modeling \cite{Pt}, reaction-diffusion with absorption problems in Biochemistry \cite{Pt}, irreversible damage and fracture evolution analysis \cite{BBLR,BBL,OMDL} and recently in constrained stratigraphic problems in Geology \cite{AGLV3,AGLV1,AGLV2,AGV4,Va}.\\
\noindent 
Concerning the study of Barenblatt equations with a stochastic force term, a few papers have been written. To the best of our knowledge, none of them proposes the study of the coupling with a random heat equation. Let us mention \cite{AdimurthiSeamVallet}, where the authors were interested in a Barenblatt equation with stochastic coefficients. In \cite{BauzetGiacomoniVallet}, the authors proposed an existence and uniqueness result for a stochastic Barenblatt equation under Dirichlet boundary conditions in the case of additive and multiplicative It\^{o} type noise. After that, well-posedness theory for stochastic abstract problems of Barenblatt's type has been investigated in \cite{BauzetVallet}. More recently, an extension of \cite{BauzetGiacomoniVallet} has been proposed in \cite{BLM}, by considering Neumann boundary conditions  and additionally the presence of a nonlinear source term.\\

\subsection{Goal of the study}
In the study of composite or bonded structures, temperature effects in the evolution of damage at the interface can not be ignored, it is even a fundamental coupling \cite{BBL,ZC}. Additionally, the introduction of stochastic and random effects is also important from a modeling point of view in order to take into account several phenomena such as microscopic fluctuations, random forcing effects of interscale interactions. For these reasons, the aim of the present work is to study the coupling between a stochastic Barenblatt equation and a random heat one under Neumann boundary conditions.  
The idea is to extend the results of \cite{BLM} for a stochastic Barenblatt equation by proposing an existence and uniqueness result for the coupled system. \\

\subsection{General notations}
\noindent For the sake of clarity, let us make precise some useful notations :
\begin{description}
\item $Q= (0,T)\times D$.
\item $x.y$ the usual scalar product of $x$ and $y$ in $\R^{d}$.
\item  $\mathscr{D}(D)=\mathscr{C}^{\infty}_c(D)$ and $\mathscr{D}'(D)$ the space of distributions on $D$.
\item  $||.||$ and $(.,.)$ respectively the usual norm and the scalar product in $L^2(D)$.
\item $E[.]$ the expectation, \textit{i.e.} the integral over $\Omega$ with respect to the probability measure $P$.
\item $\di C_{\alpha}>0$ the Lipschitz constant of $\ae$.
\item $\di \bar{C}_{\alpha}>0$ the coerciveness constant of $\ae$ which satisfies for any $x,y$ in $\R$, 
\begin{equation*}\label{coerci}
\big(\ae(x)-\ae(y)\big)(x-y)\geqslant \bar{C}_{\alpha}(x-y)^2.
\end{equation*}
\item $\bar{C}_{\tilde{\ae}}>0$  the coerciveness constant of $\tilde{\ae}$ which satisfies for any $x,y$ in $\R$, 
\begin{equation*}\label{coerci1}
\big(\tilde{\ae}(x)-\tilde{\ae}(y)\big)(x-y)\geqslant \bar{C}_{\tilde{\alpha}}(x-y)^2.
\end{equation*}
\end{description}

\subsection{Concept of solution and main results of the paper}
Let us introduce the concept of solutions we are interested in for System (\ref{eqc1}) and System (\ref{eqc123}).

\begin{defn}\label{ac1} Any pair of predictable processes  $(\vartheta,\chi)\in \left(\mathcal{N}^{2}_{w}(0,T,H^{1}(D))\right)^2$ such that $\vartheta\in L^{2}\big(\Omega, H^1(Q)\big)$ and $\partial_t(\chi-\int_0^.hdw)\in L^2(\Omega\times Q)$, is a solution of System (\ref{eqc1}) if $t$-almost everywhere in $(0,T)$, $P$-almost surely in $\Omega$, the following variational formulation holds: for any $v\in H^{1}(D),$
\begin{align}
\di\int_{D}\partial_t\vartheta v dx+\int_{D}\partial_t(\chi-\int_0^. h dw)vdx+\int_D\nabla \vartheta.\nabla vdx=0\label{dv1}\\
\di\int_{D}\tilde{\ae}\Big( \partial_t(\chi-\int_0^. h dw)\Big)vdx+\int_D\nabla \chi.\nabla vdx=\int_D\vartheta vdx,\label{dv2}
\end{align}
with $\chi(0,.)=\chi_{0}\in H^{1}(D)$ and $\vartheta(0,.)=\vartheta_{0}\in H^{1}(D)$.
\end{defn}

\begin{defn}\label{ac12} Any pair of predictable processes $(\vartheta,\chi)\in \left(\mathcal{N}^{2}_{w}(0,T,H^{1}(D))\right)^2$ such that $\vartheta\in L^{2}\big(\Omega, H^1(Q)\big)$ and $\partial_t(\chi-\int_0^. \mathscr{H}(\chi)dw)\in L^2(\Omega\times Q)$, is a solution of System (\ref{eqc123}) if $t$-almost everywhere in $(0,T)$, $P$-almost surely in $\Omega$, the following variational formulation holds: for any $v\in H^{1}(D),$
\begin{align*}\label{dv}
\di\int_{D}\partial_t\vartheta v dx+\int_{D}\partial_t(\chi-\int_0^. \mathscr{H}(\chi)dw)vdx+\int_D\nabla \vartheta.\nabla vdx=0\\
\di\int_{D}\tilde{\ae}\Big( \partial_t(\chi-\int_0^. \mathscr{H}(\chi)dw)\Big)vdx+\int_D\nabla \chi.\nabla vdx=\int_D\vartheta vdx,
\end{align*}
with $\chi(0,.)=\chi_{0}\in H^{1}(D)$ and $\vartheta(0,.)=\vartheta_{0}\in H^{1}(D) $.
\end{defn}

 \begin{rem}\label{cid} We will see later on that  the respective solutions of (\ref{eqc1}) and (\ref{eqc123}) belong to the space $L^{2}\big(\Omega, \mathscr{C}([0,T], L^{2}(D))\big)$. Thus,  they satisfy the initial condition in the following sense:
\begin{eqnarray*}
\text{P-a.s, in }\Omega \quad \quad \chi(t=0,.)=\lim_{t\rightarrow 0}\chi(t,.)\text{ in }L^{2}(D)\\ \text{and} \quad\text{P-a.s, in }\Omega\quad \quad\vartheta(t=0,.)=\lim_{t\rightarrow 0}\vartheta(t,.)\text{ in }L^{2}(D).
\end{eqnarray*}
\end{rem}

\noindent The results we want to prove in the sequel are the following:
\begin{thm}\label{EUc1}
Under assumptions $H_1$ to $H_3$, there exists a unique pair $(\vartheta,\chi)$ solution of System (\ref{eqc1}) in the sense of Definition \ref{ac1}.
\end{thm}
Moreover, the unique solution of (\ref{eqc1}) satisfies the following stability result which asserts the continuous dependence of the solution with respect to the integrand $h$ in the stochastic noise: 
\begin{prop}\label{contdep} Consider $h,\hat h$ in $\mathcal{N}^{2}_{w}(0,T,H^{1}(D))$ and denote by $(\vartheta,\chi)$ and $(\hat\vartheta,\hat\chi)$ the associated solutions to the System (\ref{eqc1}) in the sense of Definition \ref{ac1} with the respective set of data $(\vartheta_0, \chi_0, h)$ and $(\vartheta_0, \chi_0, \hat h)$. Then, there exists a constant $C_{\alpha}^T>0$, which only depends on $T$, $C_{\alpha}$ and $\bar{C}_{\alpha}$ such that for any $t$ in $[0,T]$
 \begin{align*}
& \ E\left[||(\vartheta-\hat\vartheta)(t)||^2\right]+E\left[||\nabla (\vartheta- \hat\vartheta)(t)||^2\right]\nonumber\\
&+\ \frac{1}{4}E\left[||(\chi-\hat\chi)(t)||^2\right]+\frac{1}{4}E\left[\|\nabla(\chi-\hat\chi)(t)\|^{2}\right]\\
\leq&\ {C}_{\alpha}^T\left( ||h-\hat h||^2_{L^2(\Omega\times Q_t)}+||\nabla(h- \hat h) ||^2_{L^2(\Omega\times Q_t)}\right),
 \end{align*}
 where $Q_{t}=(0,t)\times D$.
\end{prop}
 \begin{thm} \label{Non} Under Assumptions $H_2$ to $H_4$, there exists a unique pair $(\vartheta, \chi)$  solution of Problem (\ref{eqc123}) in the sense of Definition \ref{ac12}.
\end{thm}

\subsection{Outline of the paper}
\noindent The paper is organized as follows. Firstly, we are interested in showing the existence of a couple ($\vartheta, \chi$) solution of System (\ref{eqc1}). To do so, the approach is the following one: we approximate our additive stochastic system by using a semi-implicit time discretization scheme with a parameter $\Delta t>0$.\\
After deriving stability estimates satisfied by the time approximations of the couple ($\vartheta, \chi$), our aim is to pass to the limit on the obtained discrete system with respect to the time step $\Delta t$. Note that due to the random variable, classical results of compactness do not hold, and the main difficulty is in the identification of the nonlinear term's limit associated with the discretization of $\alpha\big(\partial_t(\chi-\int_0^.hdw)\big)$. Using arguments on maximal monotone operators, one is able to handle this difficulty.\\
Secondly, the uniqueness result for (\ref{eqc1}) is proven by using classical energy estimates well known for the heat equation and adapted to the random and stochastic case. This allows us to show additionally at the limit on the discretization parameter $\Delta t$ that the couple ($\vartheta, \chi$) depends continuously on the data.
\noindent Finally, exploiting this stability result of the solution with respect to the data, we are able to extend (thanks to a fixed point argument) our result of existence and uniqueness to the multiplicative case, that is the well-posedness of Problem (\ref{eqc123}).\\

\section{Time approximation of the additive case }\label{Implicittimec1}
The result of existence of a solution for Problem (\ref{eqc1}) is based on an implicit time discretization scheme for the deterministic part and an explicit one for the It\^o part. To do so, let us introduce notations used for the discretization procedure.

\subsection{Notations and preliminary results}
We consider $X$ a separable Banach space, $N\in \N^{*}$, set $\h=\frac{T}{N}$ and $t_n=n\h$ with $n\in\{0,...,N\}$. For any sequence $(x_n)_{0\leq n \leq N}\subset X$, let us denote
\begin{eqnarray*}
 x^{\Delta t}&=&\displaystyle \sum^{N-1}_{k=0}x_{k+1}\mathds{1}_{[\tk,\tkp)},\\
  x_{\Delta t}&=&\displaystyle \sum^{N-1}_{k=0}x_{k}\mathds{1}_{[\tk,\tkp)}=  x^{\Delta t}(.-\Delta t),\\
 \tilde{x}^{\Delta t}&=&\displaystyle \sum^{N-1}_{k=0}\left[\frac{x_{k+1}-x_{k}}{\Delta t}\left(.-\tk\right)+x_{k}\right]\mathds{1}_{[\tk,\tkp)}, \\
\displaystyle\frac{\partial\tilde{x}^{\Delta t}}{\partial t}&=&\displaystyle \sum^{N-1}_{k=0}\frac{x_{k+1}-x_{k}}{\Delta t}\mathds{1}_{[\tk,\tkp)},
 \end{eqnarray*}
with the convention that $t_{-1}=-\h$, for $t<0$, $\tilde{x}^{\Delta t}(t_0)=x_0$ and $x^{\Delta t}(t_N)=\tilde{x}^{\Delta t}(t_N)=x_N$. Elementary calculations yield for an arbitrary constant $C>0$ independent of $\h$
\begin{eqnarray*}
&&\displaystyle\|x^{\Delta t}\|^{2}_{L^{2}(0,T;X)}=\Delta t \displaystyle \sum^{N}_{k=1}\|x_{k}\|^{2}_{X}\quad;\quad \|\tilde{x}^{\Delta t}\|^{2}_{L^{2}(0,T;X)}\leq C\Delta t \displaystyle \sum^{N}_{k=0}\|x_{k}\|^{2}_{X}; \\
&&\displaystyle\|x^{\Delta t}-\tilde{x}^{\Delta t}\|^{2}_{L^{2}(0,T;X)}=\Delta t\displaystyle \sum^{N-1}_{k=0}\|x_{k+1}-x_{k}\|^{2}_{X};\\
&&\displaystyle\|x^{\Delta t}(.-\Delta t)-{x}^{\Delta t}\|^{2}_{L^{2}(0,T;X)}=\Delta t\displaystyle \sum^{N-1}_{k=0}\|x_{k+1}-x_{k}\|^{2}_{X}; \\
&&\displaystyle\Big\|\frac{\partial\tilde{x}^{\Delta t}}{\partial t}\Big\|^{2}_{L^{2}(0,T;X)}=\displaystyle\frac{1}{\Delta t}\displaystyle \sum^{N-1}_{k=0}\|x_{k+1}-x_{k}\|^{2}_{X}; \\
&&\displaystyle\|x^{\Delta t}\|_{L^{\infty}(0,T;X)}=\displaystyle\max_{k=1,..,N} \|x_{k}\|_{X}\text{ and } \|\tilde{x}^{\Delta t}\|_{L^{\infty}(0,T;X)}=\displaystyle\max_{k=0,..,N}\|x_{k}\|_{X}.
\end{eqnarray*}
We will use the following notations for the discretization of the data for any $n$ in $\{0,...,N\}:$
\begin{align*}
&w_n=w(\tn), \ h_n=\frac{1}{\h}\int_{\tnm}^{\tn}h(s,.)ds, \ B_{n}=\sum_{k=0}^{n-1}(w_{k+1}-w_{k})h_{k},\\
\end{align*}
with the convention that $t_{-1}=-\Delta t$ and $h(s,.)=0$ if $s<0$.
\begin{rem}
As $h$ is predictable with values in $H^1(D)$ then $h_{n}$ belongs to \\$L^2\left((\Omega,\Fn);H^1(D)\right)$ for any $n$ in $\{0,...,N\}$.
\end{rem}

\begin{rem}
For any $n$ in $\{0,...,N\}$, $B_{n}=\di \int_{0}^{t_{n}}h_{\Delta t}(s)dw(s)$.\\ Indeed, as $h_k$ is $\mathcal{F}_{t_{k}}$-measurable, one has
$$B_{n}=\sum_{k=0}^{n-1}\int_{t_{k}}^{t_{k+1}}h_kdw(s)=\int_{0}^{t_{n}}\sum_{k=0}^{n-1}h_k\mathds{1}_{[t_{k}, t_{k+1}[}(s)dw(s)=\int_{0}^{t_{n}}h_{\Delta t}(s)dw(s).$$
\end{rem}

\begin{lem}\label{hlemma1c1} There exists a constant $C\geqslant 0$ independent of $\Delta t$ such that for any $n$ in $\{0,...,N\}$
\begin{equation*}
E\left[\sum_{k=0}^{n}\|h_{k}\|_{H^1(D)}^{2}\right]\leq \frac{C}{\h}.\quad 
\end{equation*}
\end{lem}
\begin{proof} The proof of this result can be found in \cite{BLM} (Lemma 2.3).
\end{proof}

\begin{lem}\label{CVHG} The sequence $(h_{\h})$ converges to $h$ in $\mathcal{N}^{2}_{w}(0, T, \V)$ as the time discretization parameter $\Delta t$ tends to $0$.
\end{lem}

\begin{proof} See {\sc Simon} \cite{Simon}, Lemma 12 p.52.
\end{proof}

\begin{prop}\label{CVBc1} The sequences $(B^{\h})$ and $(\tilde{B}^{\h})$ converge to $\di\int_{0}^{.}hdw$ in \\ $L^{2}(0,T; L^{2}(\Omega, H^1(D)))$ as the time discretization parameter $\Delta t$ tends to $0$.
\end{prop}

\begin{proof} The proof of this result can be found in \cite{BLM} (Proposition 2.5).
\end{proof}

\begin{rem}\label{rcvh}If one assumes that $h$ belongs to $ \mathcal{N}^2_w(0,T,H^{2}(D))$, one shows in the same manner that $\di B^{\h}$ converges strongly to $\int_{0}^{.}hdw$ in  $L^{2}((0,T)\times \Omega, H^{2}(D))$ as $\h$ tends to $0$.
\end{rem}

\subsection{Discretization schemes}
Let $N$ be a positive integer and $n\in \{0,...,N\}$. Using the notations of the previous section, the discretization scheme for (\ref{dv1}) is the following one: for a given small positive parameter $\Delta t$, for $\vartheta_n$, $\chi_n$ in  $L^2\left((\Omega,\Fn);L^2(D)\right)$ and $\chi_{n+1}\in L^2\left((\Omega,\Fnp);L^2(D)\right)$,  our aim is to find $\vartheta_{n+1}$ in $L^2\left((\Omega,\Fnp);\V\right)$,  such that $P$-a.s in $\Omega$ and for any $v$ in $H^1(D)$
\begin{align}\label{bc1}
\di\int_{D}\Big(\frac{\vartheta_{n+1}-\vartheta_n}{\Delta t}\Big)v dx+\int_{D}\Big(\frac{\chi_{n+1}-\chi_{n}}{\h} - h_{n}\frac{w_{n+1}-w_{n}}{\h}\Big)vdx\nonumber\\+\int_D\nabla \vartheta_{n+1}.\nabla vdx=0.
\end{align}
Similarly the discretization scheme for (\ref{dv2}) is the following one: for a given small positive parameter $\Delta t$, for $\chi_n\in L^2\left((\Omega,\Fn);L^2(D)\right)$ and $\vartheta_{n+1}\in L^2\big((\Omega,\Fnp); L^2(D)\big)$,  our  aim is  to  find $\chi_{n+1}\in L^2\left((\Omega,\Fnp); \V\right)$,  such that $P$-a.s in $\Omega$ and for any $v$ in $H^1(D)$
\begin{align}\label{bc2}
\di\int_{D}\tilde{\ae}\Big(\frac{\chi_{n+1}-\chi_{n}}{\h} - h_{n}\frac{w_{n+1}-w_{n}}{\h}\Big)vdx+\int_D\nabla \chi_{n+1}.\nabla vdx\nonumber\\
=\int_D\vartheta_{n+1}vdx.
\end{align}
With the notations introduced in the previous section, we propose the following discretization of the variational problems (\ref{dv1}) and (\ref{dv2}) : t-almost everywhere in $(0,T)$, $P$-almost surely in $\Omega$ and for any $v$ in $H^1(D)$
\begin{align}
\int_{D}\partial_t(\tilde{\vartheta}^{\h})vdx+
\int_D\partial_t\big(\tilde{\chi}^{\h}-\tilde{B}^{\h}\big)vdx+\int_D\nabla \vartheta^{\h}.\nabla vdx=0\label{dc1}\\
\int_{D}\tilde{\ae}\l(\partial_t\big(\tilde{\chi}^{\h}-\tilde{B}^{\h}\big)\r)vdx+\int_D\nabla \chi^{\h}.\nabla vdx=\int_D\vartheta^{\h}vdx.\label{dc1bis}
\end{align}
Firstly, we show that the discrete system  composed by the approximation schemes (\ref{bc1}) and (\ref{bc2}) is well-defined. Secondly, our aim is to derive boundedness results for the approximate sequences $\vartheta^{\h}$, $\tilde{\vartheta}^{\h}$, $\chi^{\h}$ and $\tilde{\chi}^{\h}-\tilde{B}^{\h}$.

\begin{prop}\label{p1} Set $N\in \N^*, n\in\{0,...,N\}$ and $\vartheta_n, \chi_n\in L^2\left((\Omega,\Fn);L^2(D)\right)$, $\chi_{n+1}\in L^2\left((\Omega,\Fnp);L^2(D)\right)$. If we assume that $\Delta t\leq1$, then
there exists a unique $\vartheta_{n+1}\in L^2\left((\Omega,\Fnp);\V\right)$ satisfying (\ref{bc1}), $P$-a.s in $\Omega$ and for any $v$ in $\V.$
\end{prop}
\begin{proof} A direct application of Lax-Milgram Theorem gives us the result.
\end{proof}

\begin{prop}\label{p2} Set $N\in \N^*, n\in\{0,...,N\}$ and $\chi_n\in L^2\left((\Omega,\Fn);L^2(D)\right)$, $\vartheta_{n+1}\in L^2\left((\Omega,\Fnp);L^2(D)\right)$. If we assume that $\Delta t<1$, then
there exists a unique $\unp\in L^2\left((\Omega,\Fnp);\V\right)$ satisfying (\ref{bc2}), $P$-a.s in $\Omega$ and for any $v$ in $H^1(D).$
\end{prop}

\begin{proof} The proof is mostly the same as in \cite{BLM}, Proposition 2.7, so we refer the reader to this paper.
\end{proof}

\begin{prop}Set $N\in \N^*, n\in\{0,...,N\}$ and $\vartheta_n,\chi_n\in L^2\left((\Omega,\Fn);\H\right)$. Assume that $\Delta t <\bar{C}_{\tilde\alpha}$, then
there exists a unique pair $(\vartheta_{n+1},\chi_{n+1})$ belonging to $L^2\big((\Omega,\Fnp);\V\big)\times  L^2\left((\Omega,\Fnp);\V\right)$ and satisfying (\ref{bc1}-\ref{bc2}) $P$-a.s in $\Omega$ and for any $v$ in $H^1(D)$.
\end{prop}

\begin{proof} Set $N\in \N^*, n\in\{0,...,N\}$ and $\vartheta_n,\chi_n\in L^2\left((\Omega,\Fn);\H\right)$. We  introduce the following functionals
$$f:  L^2\left((\Omega,\Fnp);L^2(D)\right)\to  L^2\left((\Omega,\Fnp);\V\right)$$
$$\tilde\chi\mapsto\vartheta^f,$$
where $\vartheta^f$ satisfies, $P$-a.s in $\Omega$ and for any $v$ in $H^1(D)$
\begin{align}\label{comp1}
\di\int_{D}\Big(\frac{\vartheta^f-\vartheta_n}{\Delta t}\Big)v dx+\int_{D}\Big(\frac{\tilde\chi-\chi_{n}}{\h} - h_{n}\frac{w_{n+1}-w_{n}}{\h}\Big)vdx\nonumber\\+\int_D\nabla \vartheta^f.\nabla vdx=0.
\end{align}
Thanks to Proposition  \ref{p1}, $f$ is well defined. Similarly, we introduce
$$g: L^2\left((\Omega,\Fnp);L^2(D)\right)\to  L^2\left((\Omega,\Fnp);\V\right)$$
$$\tilde{\vartheta}\mapsto\chi^g,$$
where $\chi^g$ satisfies, $P$-a.s in $\Omega$ and for any $v$ in $H^1(D)$
\begin{align}\label{comp2}
&\di\int_{D}\tilde{\ae}\Big(\frac{\chi^g-\chi_{n}}{\h} - h_{n}\frac{w_{n+1}-w_{n}}{\h}\Big)vdx+\int_D\nabla\chi^g.\nabla vdx=\int_D\tilde\vartheta vdx.
\end{align}
Thanks to Proposition  \ref{p2}, $g$ is well defined. Let us prove that the  composition $g\circ f$ is a strict contraction in $ L^2\left((\Omega,\Fnp);L^2(D)\right)$.
On the one hand, note that (\ref{comp1}) can be written $P$-a.s in $\Omega$ and for any  $v$ in $H^1(D)$ as 
\begin{align}
 \label{bc12}
&\di \int_D\vartheta^fvdx+\Delta t\int_D\nabla \vartheta^f.\nabla vdx\nonumber\\
=&\int_D(\vartheta_n-\tilde\chi+\chi_{n}+h_{n}(w_{n+1}-w_{n}))vdx.
\end{align} 
Set $\tilde\chi_{1}, \tilde\chi_{2}$ in $L^2\left((\Omega,\Fn);\H\right)$ and define $\vartheta^f_1=f(\tilde\chi_{1})$, $\vartheta^f_2=f(\tilde\chi_{2})$. Then using (\ref{bc12}), one gets $P$-a.s in $\Omega$ and for any  $v$ in $H^1(D)$
\begin{align}
\label{bc123}
\di\int_D(\vartheta^f_1-\vartheta^f_2)vdx+\Delta t\int_D\nabla (\vartheta^f_1-\vartheta^f_2).\nabla v dx=\int_D(\tilde\chi_{2}-\tilde\chi_{1})v dx.
\end{align}
By choosing $v=(\vartheta^f_1-\vartheta^f_2)$ in (\ref{bc123}) we obtain
\begin{align*}
\di\int_D\big|\vartheta^f_1-\vartheta^f_2\big|^2dx+\Delta t\int_D|\nabla (\vartheta^f_1-\vartheta^f_2)|^2dx
=\int_D(\tilde\chi_{2}-\tilde\chi_{1})(\vartheta^f_1-\vartheta^f_2)dx.
\end{align*}
Then
\begin{align*}
\di\int_D|\vartheta^f_1-\vartheta^f_2|^2dx+2\h\int_D|\nabla (\vartheta^f_1-\vartheta^f_2)|^2dx\leq\int_D|\tilde\chi_{1}-\tilde\chi_{2}|^2dx.
\end{align*}
By taking the expectation, one gets
\begin{align*}
\di E\left[\int_D|\vartheta^f_1-\vartheta^f_2|^2dx\right]+2\h E\left[\int_D|\nabla (\vartheta^f_1-\vartheta^f_2)|^2dx\right]\leq E\left[\int_D|\tilde\chi_{1}-\tilde\chi_{2}|^2dx\right],
\end{align*}
which implies that
\begin{align}
\label{bccc}
\di E\left[||\vartheta^f_1-\vartheta^f_2||^2\right]\leq E\left[||\tilde\chi_{1}-\tilde\chi_{2}||^2\right].
\end{align}
On the other hand, by defining $\tilde\chi^g_1=g(\vartheta^f_1)$ and $\tilde\chi^g_2=g(\vartheta^f_2)$,  (\ref{comp2})  gives $P$-a.s in $\Omega$ and for any $v$ in $H^1(D)$
\begin{align}\label{bc222}
\di&\int_D\left[\tilde{\ae}\Big(\frac{\tilde\chi^g_1-\chi_{n}}{\h} - h_{n}\frac{w_{n+1}-w_{n}}{\h}\Big)-\tilde{\ae}\Big(\frac{\tilde\chi^g_2-\chi_{n}}{\h} - h_{n}\frac{w_{n+1}-w_{n}}{\h}\Big)\right]vdx\nonumber\\
&+\int_D\nabla (\tilde\chi^g_1-\tilde\chi^g_2).\nabla vdx=\int_D(\vartheta^f_1-\vartheta^f_2)vdx.
\end{align}
By choosing $v=\tilde\chi^g_1-\tilde\chi^g_2$ in (\ref{bc222}) and using the coercivity of $\tilde\ae$, we obtain by taking the expectation
\begin{align*}
\di&\Big(\frac{\bar{C}_{\tilde{\ae}}}{\h}-\frac{1}{2}\Big)E\left[\int_D|\tilde\chi^g_1-\tilde\chi^g_2|^2dx\right]\\
+&E\left[\int_D|\nabla (\tilde\chi^g_1-\tilde\chi^g_2)|^2dx\right]\leq \frac{1}{2}E\left[\int_D|\vartheta^f_1-\vartheta^f_2|^2dx\right],
\end{align*} 
and so
\begin{align}
\label{bc222221}
\di2\Big(\frac{\bar{C}_{\tilde{\ae}}}{\Delta t}-\frac{1}{2}\Big)E\left[||\tilde\chi^g_1-\tilde\chi^g_2||^2\right]\leq E\left[||\vartheta^f_1-\vartheta^f_2||^2\right].
\end{align}
By comparing (\ref{bccc}) and (\ref{bc222221}), we get

\begin{align}
\label{bc2222211}
\di E\left[||\tilde\chi^g_1-\tilde\chi^g_2||^2\right]\leq\frac{1}{2(\frac{\bar{C}_{\tilde{\ae}}}{\Delta t}-\frac{1}{2})} E\left[||\tilde\chi_1-\tilde\chi_2||^2\right].
\end{align}
Under the assumption $\Delta t <\bar{C}_{\tilde\alpha}$, the function $g\circ f$ which maps $ L^2\big((\Omega,\Fnp); L^2(D)\big)$ in itself is a strict contraction and admits a unique fixed point in \\$L^2\l((\Omega,\Fnp);L^2(D)\r)$. Using this, there exists a unique pair $(\vartheta_{n+1},\chi_{n+1})$ in $L^2\big((\Omega,\Fnp); \V\big)\times  L^2\left((\Omega,\Fnp);\V\right)$ satisfying (\ref{bc1}) and (\ref{bc2}) $P$-a.s in $\Omega$ and for any $v$ in $H^1(D)$.
\end{proof}

\subsection{First estimates on the approximate sequences}
Our aim is to find boundedness results for the sequences $\tilde{\vartheta}^{\h}$,$\vartheta^{\h}$,$\tilde{\chi}^{\h}$, $\chi^{\h}$ and $\tilde{\chi}^{\h}-\tilde{B}^{\h}$.

\begin{prop}\label{fae}
There exists a constant $C>0$ independent of $\h$ such that
\begin{align}
|| \tilde{\vartheta}^{\h}||_{L^\infty(0,T,L^2(\Omega\times D))},|| \vartheta^{\h}||_{L^\infty(0,T;L^2(\Omega\times D))} &\leq C, \label{2c13}\\ 
||  \tilde{\vartheta}^{\h}-\vartheta^{\h}||_{L^2(\Omega\times Q)} &\leq C\sqrt{\Delta t}, \label{2c15}\\
|| \nabla \tilde{\vartheta}^{\h}||_{L^2(\Omega\times Q)},|| \nabla \vartheta^{\h}||_{L^2(\Omega\times Q)} &\leq C, \label{2c12}\\
|| \nabla \tilde{\chi}^{\h}||_{L^\infty(0,T,L^2(\Omega\times D))},|| \nabla \chi^{\h}||_{L^\infty(0,T;L^2(\Omega\times D))} &\leq C, \label{2c1}\\
|| \partial_t(\tilde{\chi}^{\h}-\tilde{B}^{\h}) ||_{L^2(\Omega\times Q)} &\leq C.\label{1c1}
\end{align}
\end{prop}

\begin{proof} Set $N\in \N^*$, $n\in \{0,..,N-1\}$ and $k\in \{0,...,n\}$. \smallskip Consider the variational formulations $(\ref{bc1})$ and $(\ref{bc2})$ with  the couple of indexes $(k+1,k)$. By adding $(\ref{bc1})$ with  the test function $v=\di \vartheta_{k+1}$ and (\ref{bc2}) with the test function $v=\di\frac{\chi_{k+1}-\chi_{k}}{\h}  -h_{k} \frac{w_{k+1}-w_{k}}{\h} $, one gets 
\begin{align*}
&\int_D(\frac{\vartheta_{k+1}-\vartheta_{k}}{\h})\vartheta^{k+1}dx+\l\|\nabla\vartheta_{k+1}\r\|^2+
\l\|\frac{\chi_{k+1}-\chi_{k}}{\h}  - h_{k}\frac{w_{k+1}-w_{k}}{\h}\r\|^2\\
&+\int_D \ae\l(\di\frac{\chi_{k+1}-\chi_{k}}{\h}  - h_{k}\frac{w_{k+1}-w_{k}}{\h} \r) \times \l(\frac{\chi_{k+1}-\chi_{k}}{\h}  - h_{k}\frac{w_{k+1}-w_{k}}{\h} \r)  \,dx\\
&+\int_D\nabla \chi_{k+1}.\nabla\l(\frac{\chi_{k+1}-\chi_{k}}{\h}  - h_{k}\frac{w_{k+1}-w_{k}}{\h}\r)dx= 0.
\end{align*}
Using the coerciveness property of $\ae$, one gets
\begin{align*}
&\int_D\vartheta^{k+1}(\frac{\vartheta_{k+1}-\vartheta_{k}}{\h})dx+\l\|\nabla\vartheta_{k+1}\r\|^2+
\di (\bar{C}_{\alpha}+1)\l \|\frac{\chi_{k+1}-\chi_{k}}{\h}  -h_{k} \frac{w_{k+1}-w_{k}}{\h} \r \|^2 \\
&+  \int_D\nabla \chi_{k+1}.\nabla\l(\frac{\chi_{k+1}-\chi_{k}}{\h}  - h_{k}\frac{w_{k+1}-w_{k}}{\h} \r)dx\leq 0.
\end{align*}
Then
\begin{align*}
&\int_D\vartheta^{k+1}(\frac{\vartheta_{k+1}-\vartheta_{k}}{\h})dx+\l\|\nabla\vartheta_{k+1}\r\|^2+
 \di (\bar{C}_{\alpha}+1)\l \|\frac{\chi_{k+1}-\chi_{k}}{\h}  -h_{k} \frac{w_{k+1}-w_{k}}{\h} \r \|^2 \\
 &+ \int_D\nabla \chi_{k+1}.\nabla\frac{\chi_{k+1}-\chi_{k}}{\h}dx \\
\leq&\di \int_D\nabla(\chi_{k+1}-\chi_{k}).\nabla h_{k}\frac{w_{k+1}-w_{k}}{\h} dx +\int_D\nabla \chi_{k}.\nabla h_{k}\frac{w_{k+1}-w_{k}}{\h} dx.
\end{align*}
Using the formula $$a(a-b)=\frac{1}{2}\{a^2-b^2+(a-b)^2\}$$ with $a=\vartheta_{k+1}$ (respectively $a=\nabla\chi_{k+1}$) and $b=\vartheta_{k}$ (respectively $b=\nabla\chi_{k}$), one gets for any $\epsilon>0$
\begin{align*}
&\frac{1}{2\h}\l[||\vartheta_{k+1}||^2-||\vartheta_{k}||^2 + ||\vartheta_{k+1}-\vartheta_{k}||^2\r]+||\nabla \vartheta_{k+1}||^2\\
&+(\bar{C}_{\alpha}+1)\l\|\frac{\chi_{k+1}-\chi_{k}}{\h}  - h_{k}\frac{w_{k+1}-w_{k}}{\h} \r\|^2   \\
&+\frac{1}{2\h}\l[||\nabla \chi_{k+1}||^2-||\nabla \chi_{k}||^2 + ||\nabla (\chi_{k+1}-\chi_{k})||^2\r] \\
\leq&\ \frac{\epsilon}{2\h}||\nabla (\chi_{k+1}-\chi_{k})||^2 + \frac{|w_{k+1}-w_{k}|^2}{2\epsilon \h} ||\nabla h_{k}||^2\\
&- \int_D\nabla \chi_{k}.\nabla h_{k}\frac{w_{k+1}-w_{k}}{\h} dx.
\end{align*}
Then, since $\nabla\chi_{k}$ and $\nabla h_{k}$ are $\mathcal{F}_{t_{k}}$-measurable, by taking the expectation one gets
\begin{align*}
&\frac{1}{2\h}E\l[\l\|\vartheta_{k+1}\r\|^2-\l\|\vartheta_{k}\r\|^2 + \l\|\vartheta_{k+1}-\vartheta_{k}\r\|^2\r]+E\l[\l\|\nabla \vartheta_{k+1}\r\|^2\r]\\
&+(\bar{C}_{\alpha}+1)E\l[\l\|\frac{\chi_{k+1}-\chi_{k}}{\h}  -h_{k} \frac{w_{k+1}-w_{k}}{2\h} \r\|^2\r]  \\
&+ \frac{1}{2\h}E\left[\l\| \nabla \chi_{k+1}\r\|^2-\l\|\nabla \chi_{k}\r\|^2 + \l\|\nabla (\chi_{k+1}-\chi_{k})\r\|^2\right]\\
\leq&\  \frac{\epsilon}{2\h}E\l[||\nabla (\chi_{k+1}-\chi_{k})||^2\r] +\frac{1}{2\epsilon}E\l[ ||\nabla h_{k}||^2\r].
\end{align*}
In this way
\begin{align*}
&E\l[\l\|\vartheta_{k+1}\r\|^2\r]-E\l[\l\|\vartheta_{k}\r\|^2\r] + E\l[\l\|\vartheta_{k+1}-\vartheta_{k}\r\|^2\r]+2\h E\l[\l\|\nabla \vartheta_{k+1}\r\|^2\r]\\
&+2(\bar{C}_{\alpha}+1)\h E\l[\l\|\frac{\chi_{k+1}-\chi_{k}}{\h}  - h_{k}\frac{w_{k+1}-w_{k}}{\h} \r\|^2\r]  \\
&+ E\l[\l\|\nabla \chi_{k+1}\r\|^2\r]-E\l[\l\|\nabla \chi_{k}\r\|^2\r]
+(1-\epsilon)E\l[\l\|\nabla (\chi_{k+1}-\chi_{k})\r\|^2\r] \\
\leq&\ \frac{\h}{\epsilon} E\l[\l\|\nabla h_{k}\r\|^2\r].
\end{align*}
By summing from $k=0$ to $n$, one gets
\begin{align*}
&\sum_{k=0}^{n}E\l[\l\|\vartheta_{k+1}\r\|^2\r]-\sum_{k=0}^{n}E\l[\l\|\vartheta_{k}\r\|^2\r] + \sum_{k=0}^{n}E\l[\l\|\vartheta_{k+1}-\vartheta_{k}\r\|^2\r]+2\sum_{k=0}^{n}\h E\l[\l\|\nabla \vartheta_{k+1}\r\|^2\r]\\
&+2(\bar{C}_{\alpha}+1)\sum_{k=0}^{n}\h E\l[\l\|\frac{\chi_{k+1}-\chi_{k}}{\h}  - \frac{w_{k+1}-w_{k}}{\h} h_{k}\r\|^2\r]  \\
&+ \sum_{k=0}^{n}E\l[\l\|\nabla \chi_{k+1}\r\|^2\r]- \sum_{k=0}^{n}E\l[\l\|\nabla \chi_{k}\r\|^2\r]+(1-\epsilon)\sum_{k=0}^{n}E\l[\l\|\nabla (\chi_{k+1}-\chi_{k})\r\|^2\r] \\
\leq&\ \frac{1}{\epsilon} \sum_{k=0}^{n}\h E\l[\l\|\nabla h_{k}\r\|^2\r].
\end{align*}
By taking $\di\epsilon=\frac{1}{2}$ and using Lemma \ref{hlemma1c1}, there exists a constant $C>0$ independent of $\h$ such that

\begin{align}
&E\l[\l\|\vartheta_{n+1}\r\|^2\r] + \sum_{k=0}^{n}E\l[\l\|\vartheta_{k+1}-\vartheta_{k}\r\|^2\r]+\sum_{k=0}^{n}\h E\l[\l\|\nabla \vartheta_{k+1}\r\|^2\r]\nonumber\\
&+\sum_{k=0}^{n}\h E\l[\l\|\frac{\chi_{k+1}-\chi_{k}}{\h}  - h_k \frac{w_{k+1}-w_{k}}{\h}\r\|^2\r]  +E\l[\l\|\nabla \chi_{n+1}\r\|^2\r]\nonumber\\
&+\frac{1}{2} \sum_{k=0}^{n}E\l[\l\|\nabla (\chi_{k+1}-\chi_{k})\r\|^2\r] \leq C, \label{estimation2c1}
\end{align}
and we get directly the announced estimates:
\begin{eqnarray*}
||  \tilde{\vartheta}^{\h}||_{L^\infty(0,T,L^2(\Omega\times D))},||  \vartheta^{\h}||_{L^\infty(0,T,L^2(\Omega\times D))} &\leq& C,\\
||  \tilde{\vartheta}^{\h}-\vartheta^{\h} ||_{L^2(\Omega\times Q)} &\leq& C\sqrt{\h},\\
|| \nabla\tilde{\vartheta}^{\h} ||_{L^2(\Omega\times Q)},|| \nabla{\vartheta}^{\h} ||_{L^2(\Omega\times Q)} &\leq& C.\\
|| \partial_t(\tilde{\chi}^{\h}-\tilde{B}^{\h}) ||_{L^2(\Omega\times Q)} &\leq& C,\\
|| \nabla \tilde{\chi}^{\h}||_{L^\infty(0,T,L^2(\Omega\times D))},|| \nabla \chi^{\h}||_{L^\infty(0,T,L^2(\Omega\times D))} &\leq& C.
\end{eqnarray*}
Additionally, let us note that one can also deduce from (\ref{estimation2c1}) the following bound:
\begin{align}
|| \nabla (\tilde{\chi}^{\h}-\chi^{\h}) ||_{L^2(\Omega\times Q)} &\leq C\sqrt{\h}.\label{ng}
\end{align}
\end{proof}
\noindent From these first estimates, one can deduce directly the following ones. 
\begin{prop}\label{fbae}
There exists a constant $C>0$ independent of $\h$ such that
\begin{align}
|| \tilde{\chi}^{\h}-\chi^{\h} ||_{L^2((0,T)\times\Omega, H^1(D))}&\leq C\sqrt{\h},\label{utmu}\\
|| \tilde{\vartheta}^{\h}-\tilde{\vartheta}^{\h}(.-\Delta t) ||_{L^2(\Omega \times Q)} &\leq C\sqrt{\h},\label{utmubis1}\\
|| \tilde{\chi}^{\h}-\tilde{\chi}^{\h}(.-\Delta t) ||_{L^2((0,T)\times\Omega, H^1(D))} &\leq C\sqrt{\h},\label{utmubis}\\
|| \tilde{\chi}^{\h}-\tilde{B}^{\h} ||_{L^{\infty}(0,T;L^2(\Omega\times D))} &\leq C,\label{6c1}\\
|| \nabla (\tilde{\chi}^{\h}-\tilde{B}^{\h}) ||_{L^2(\Omega\times Q)}&\leq C, \label{5c1}\\
||\tilde{\chi}^{\h}||_{L^2((0,T)\times \Omega,H^1(D))}, ||\chi^{\h}||_{L^2((0,T)\times \Omega,H^1(D))}&\leq C,\label{butuut}\\
||\tilde{\chi}^{\h}(.-\h)||_{\mathcal{N}^2_w(0,T,H^1(D))}, \ ||\tilde{\vartheta}^{\h}(.-\h)||_{\mathcal{N}^2_w(0,T,H^1(D))}&\leq C. \label{butuutbis}
\end{align}
\end{prop}
\begin{proof}
Using (\ref{estimation2c1}), we have
\begin{align*}
&\l\| \tilde{\chi}^{\h}-\chi^{\h} \r\|^2_{L^2(\Omega\times Q)}= \Delta t\sum_{k=0}^{N-1}E\l[\l\|\chi_{k+1}-\chi_k\r\|^2\r]\\
\leq&\ \Delta t\sum_{k=0}^{N-1}E\l[2\Delta t^2\l\|\frac{\chi_{k+1}-\chi_{k}}{\h}  - h_k \frac{w_{k+1}-w_{k}}{\h}\r\|^2\hspace*{-0.1cm}+\hspace*{-0.05cm}2\l\|h_k(w_{k+1}-w_{k})\r\|^2    \r]\\
=&\ 2\Delta t^2\sum_{k=0}^{N-1}\Delta t E\l[\l\|\frac{\chi_{k+1}-\chi_{k}}{\h}  - h_k \frac{w_{k+1}-w_{k}}{\h}\r\|^2\r]\hspace*{-0.05cm}+\hspace*{-0.05cm}2\sum_{k=0}^{N-1}\Delta t^2 E\l[ \l\|h_k\r\|^2\r]\\
\leq&\ C \Delta t, 
\end{align*}
combining it  with the previous estimate  (\ref{ng}), one deduces that  (\ref{utmu}) holds. 
Note that 
$$ \tilde{\vartheta}^{\h}-\tilde{\vartheta}^{\h}(.-\Delta t) =\big(\tilde{\vartheta}^{\h}-{\vartheta}^{\h}\big)+\big(\vartheta^{\h}-{\vartheta}^{\h}(.-\Delta t)\big)+\big({\vartheta}^{\h}(.-\Delta t)-\tilde{\vartheta}^{\h}(.-\Delta t)\big)
$$
and that there exists a constant $C>0$ independent of $\h$ such that 
$$||\tilde{\vartheta}^{\h}-{\vartheta}^{\h}||_{L^2(\Omega \times Q)}^2, ||\vartheta^{\h}-{\vartheta}^{\h}(.-\Delta t)||_{L^2(\Omega \times Q)}^2\text{ and }||{\vartheta}^{\h}(.-\Delta t)-\tilde{\vartheta}^{\h}(.-\Delta t)||_{L^2(\Omega \times Q)}^2$$
are controlled by
$$ C\h\sum_{k=0}^{N-1}E\left[||\vartheta_{k+1}-\vartheta_k||^2\right].$$
Then, owing to (\ref{estimation2c1}), one gets directly (\ref{utmubis1}).
Now, using the same kind of decomposition for $\tilde{\chi}^{\h}-\tilde{\chi}^{\h}(.-\h)$, one shows that (\ref{utmubis}) holds.
Additionally for any $n$ in $\{0,...,N-1\}$ since $B_0=0$ one has
\begin{align*}
&E\l[\l\|\chi_{k+1}-\bnp\r\|^2\r]\\
\leq&\ 2||\chi_0||^2+2T\sum_{k=0}^n\Delta tE\l[\l|\l|\frac{\chi_{k+1}-\chi_{k}}{\h}  - h_k \frac{w_{k+1}-w_{k}}{\h}\r|\r|^2\r]
\end{align*}
combining this with (\ref{estimation2c1}), we show that $|| \tilde{\chi}^{\h}-\tilde{B}^{\h} ||_{L^{\infty}(0,T;L^2(\Omega\times D))} \leq C$.
Let us now prove that $\di || \nabla (\tilde{\chi}^{\h}-\tilde{B}^{\h}) ||_{L^2(\Omega\times Q)}$ is bounded independently of $\Delta t$. Using (\ref{2c1}),  it remains to show that $\nabla \tilde{B}^{\h}$ is bounded in $L^2(\Omega\times Q)$. Due to  Lemma \ref{hlemma1c1} and the fact that  $E\left[(w_{j+1}-w_{j})^{2}\right]=\h$ for any $j\in\{0,...,N-1\}$, one has
\begin{equation*}
\begin{array}{lll}
|| \nabla \tilde{B}^{\h}||^{2}_{L^2(\Omega\times Q)}
&\leq&\h\di \sum_{k=0}^{N}E\l[\int_{D}\left(\sum_{j=0}^{k}(w_{j+1}-w_{j})\nabla h_{j}\right)^{2}dx\r] \\
&=&\h\di \sum_{k=0}^{N}\sum_{j=0}^{k}\int_{D}E\left[(w_{j+1}-w_{j})^{2}\right]E\l[(\nabla h_{j})^{2}\r]dx \\
&=&\h\di \sum_{k=0}^{N} \h E\sum_{j=0}^{k}\|h_{j}\|_{H^{1}(D)}^{2}\\
&\leq&C,
\end{array}
\end{equation*}
and the result holds. Using the fact that $\tilde{\chi}^{\h}-\tilde{B}^{\h}$ and $\tilde{B}^{\h}$ are bounded in $L^2(\Omega\times Q)$, one gets that $\tilde{\chi}^{\h}$ is also bounded in  $L^2(\Omega\times Q)$. Finally, combining this with (\ref{2c1}), one obtains the boundedness of $\tilde{\chi}^{\h}$ in $L^2((0,T)\times \Omega,H^1(D))$. Thanks to (\ref{utmu})-(\ref{utmubis}), one gets the same result for $\tilde{\chi}^{\h}(.-\h)$ and $\chi^{\h}$ which gives  (\ref{butuut}).\\

Note that $\tilde{\chi}^{\h}(.-\h)$ and $\tilde\vartheta^{\h}(.-\h)$ are bounded in $L^2((0,T)\times \Omega, H^1(D))$ respectively  due  to (\ref{utmubis})-(\ref{butuut}) and (\ref{2c13})-(\ref{2c12})-(\ref{utmubis1}). Thus, they belong to \\$\mathcal{N}^2_w(0,T, H^1(D))$  as  continuous and adapted processes. Finally, (\ref{butuutbis}) holds.
\end{proof}
\pagebreak
\subsection{Second estimates on the approximate sequences}
\begin{prop}\label{sae}
There exists a constant $C>0$ independent of $\h$ such that
\begin{align}
\l\|\nabla \tilde{\vartheta}^{\h}\r\|_{L^\infty(0,T,L^2(\Omega\times D))},\l\|\nabla \vartheta^{\h}\r\|_{L^\infty(0,T;L^2(\Omega\times D))} &\leq C, \label{sp1}\\
\l\| \nabla(\tilde{\vartheta}^{\h}-\vartheta^{\h})\r\|_{L^2(\Omega\times Q)} &\leq C\sqrt{\h}, \label{sp2}\\
|| \nabla\big(\tilde{\vartheta}^{\h}-\tilde{\vartheta}^{\h}(.-\Delta t)\big) ||_{L^2(\Omega \times Q)} &\leq C\sqrt{\h}, \label{sp2bis}\\
\l\|\partial_t\tilde{\vartheta}^{\h}\r\|_{L^2(\Omega \times Q)}&\leq C.\label{sp3}
\end{align}
\end{prop}
\begin{proof} Set $N\in \N^*$, $n\in \{0,..,N-1\}$ and $k\in \{0,...,n\}$. \smallskip We consider the variational formulation $(\ref{bc1})$ with  the couple of indexes $(k+1,k)$ and choose the particular test function  $\di v=\frac{\vartheta_{k+1}-\vartheta_k}{\Delta t}$ to get $P$-almost surely in $\Omega$
\begin{align*}
&\l\|\frac{\vartheta_{k+1}-\vartheta_k}{\Delta t}\r\|^2+\int_{D}\Big(\frac{\chi_{k+1}-\chi_{k}}{\h} - h_{k}\frac{w_{k+1}-w_{k}}{\h}\Big)\times\big(\frac{\vartheta_{k+1}-\vartheta_k}{\Delta t}\big)dx\\
&+\int_D\nabla \vartheta_{k+1}.\nabla \big(\frac{\vartheta_{k+1}-\vartheta_k}{\Delta t}\big)dx=0.
\end{align*}
Then, for any $\delta>0$, we have
\begin{align*}
&\l\|\frac{\vartheta_{k+1}-\vartheta_k}{\Delta t}\r\|^2
+\frac{1}{2\h}\l[\l\|\nabla \vartheta_{k+1}\r\|^2-\l\|\nabla \vartheta_{k}\r\|^2 + \l\|\nabla (\vartheta_{k+1}-\vartheta_{k})\r\|^2\r] \\
&\leq
\frac{1}{2\delta}\l\|\frac{\chi_{k+1}-\chi_{k}}{\h} - h_{k}\frac{w_{k+1}-w_{k}}{\h}\r\|^2+\frac{\delta}{2}\l\|\frac{\vartheta_{k+1}-\vartheta_k}{\Delta t}\r\|^2.\\
\end{align*}
By choosing $\delta=1$, taking the expectation and summing from $k=0$ to $n$, one gets 
\begin{align*}
&\sum_{k=0}^{n}\h E\l[\l\|\frac{\vartheta_{k+1}-\vartheta_k}{\Delta t}\r\|^2\r]
+\sum_{k=0}^{n}E\l[\l\|\nabla \vartheta_{k+1}\r\|^2\r]-\sum_{k=0}^{n}E\l[\l\|\nabla \vartheta_{k}\r\|^2\r] \\
 &+\sum_{k=0}^{n}E\l[\l\|\nabla (\vartheta_{k+1}-\vartheta_{k})\r\|^2\r] 
\leq \sum_{k=0}^{n}\h E\l[\l\|\frac{\chi_{k+1}-\chi_{k}}{\h} - h_{k}\frac{w_{k+1}-w_{k}}{\h}\r\|^2\r].
\end{align*}
Thanks to (\ref{estimation2c1}), one concludes that
\begin{align*}
\sum_{k=0}^{n}\h E\l[\l\|\frac{\vartheta_{k+1}-\vartheta_k}{\Delta t}\r\|^2\r]+E\l[\l\|\nabla \vartheta_{n+1}\r\|^2\r]+ \sum_{k=0}^{n}E\l[\l\|\nabla (\vartheta_{k+1}-\vartheta_{k})\r\|^2\r] 
\leq C.
\end{align*}
Finally, we have directly the announced estimates
\begin{align*}
|| \nabla\tilde{\vartheta}^{\h}||_{L^\infty(0,T,L^2(\Omega\times D))},||\nabla \vartheta^{\h}||_{L^\infty(0,T;L^2(\Omega\times D))} &\leq C,\\
|| \nabla(\tilde{\vartheta}^{\h}-\vartheta^{\h})||_{L^2(\Omega\times Q)} &\leq C\sqrt{\h}, \\
\l\|\partial_t\tilde{\vartheta}^{\h}\r\|_{L^2(\Omega \times Q)}&\leq C.
\end{align*}
Arguing as in the proof of (\ref{utmubis1}), one shows finally that 
$$|| \nabla\big(\tilde{\vartheta}^{\h}-\tilde{\vartheta}^{\h}(.-\Delta t)\big) ||_{L^2(\Omega \times Q)} \leq C\sqrt{\h}.$$
\end{proof}

\subsection{Weak convergence results on the approximate sequences}

Due to Propositions \ref{fae}, \ref{fbae} and \ref{sae}, we obtain the following convergence results.
\begin{prop}\label{CVuc1}
Up to subsequences denoted in the same way, there exists  $\vartheta$  belonging to $\mathcal{N}^2_w(0,T,H^{1}(D))\cap L^{2}\big(\Omega, H^1(Q)\big)$  such that
\begin{equation*}
\begin{array}{ll}
(i)&\tilde{\vartheta^{\h}},  \vartheta^{\h} \rightharpoonup \vartheta \mbox{ in }L^2((0,T)\times\Omega,H^1(D)), \\
(ii)&\nabla \tilde{\vartheta}^{\h}, \nabla \vartheta^{\h} \wstarconverge \nabla \vartheta \mbox{ in } L^\infty(0,T;L^2(\Omega\times D)),\\
(iii)&\partial_t\tilde\vartheta^{\h}\rightharpoonup \partial_t\vartheta \mbox{ in }L^2(\Omega\times Q),\\
(iv)&\tilde{\vartheta}^{\h}(0) \di\rightharpoonup \vartheta(0)\text{ in }L^{2}(\Omega\times D).
\end{array}
\end{equation*}
\end{prop}

\begin{proof}\quad  \\
$(i)$ Thanks to (\ref{2c13}), (\ref{2c15}), (\ref{2c12}), (\ref{utmubis1}), (\ref{butuutbis}), (\ref{sp2}) and (\ref{sp2bis}), there exists $\vartheta$  in $L^2((0,T)\times\Omega,H^1(D))$
such that, up to subsequences denoted in the same way, we have
\begin{equation*}
 \tilde{\vartheta}^{\h},\vartheta^{\h}, \vartheta^{\h}(.-\h) \rightharpoonup \vartheta\text{  in }L^2((0,T)\times\Omega,H^1(D)).
\end{equation*}
Since $\tilde{\vartheta}^{\h}(.-\h)$ belongs to the Hilbert space $\mathcal{N}^2_w(0,T,H^{1}(D))$ endowed with the norm of $L^2((0,T)\times \Omega,H^1(D))$, one gets that $\vartheta$ is also in  $\mathcal{N}^2_w(0,T,H^{1}(D))$.\\
$(ii)$ Using (\ref{sp1})-(\ref{sp2}), one gets directly that up to subsequences denoted in the same way,
\begin{equation*}
\nabla \tilde{\vartheta}^{\h},\nabla \vartheta^{\h} \wstarconverge \nabla \vartheta\text{ in }L^\infty(0,T;L^2(\Omega\times D)).
\end{equation*}
$(iii)$ (\ref{sp3}) gives us directly the announced result.\\
$(iv)$ Since $L^{2}\big(\Omega, H^1(Q)\big)$ is continuously embedded in $L^2\big(\Omega,\mathscr{C}([0,T], L^{2}(D))\big)$, one gets that $\di \vartheta$ belongs to $L^2\big(\Omega,\mathscr{C}([0,T], L^{2}(D))\big)$. Thus, $\vartheta$ is an element of $\mathscr{C}([0,T],L^{2}(\Omega\times D))$ and we have
and that $$\tilde\vartheta^{\h}(0)\rightharpoonup \di \vartheta(0)\text{ in }L^{2}(\Omega\times D).$$
\end{proof}
\begin{prop}\label{CVuc1bis}
Up to subsequences denoted in the same way, there exist  $\chi$  belonging to $\mathcal{N}^2_w(0,T,H^{1}(D))\cap L^{2}\big(\Omega, \mathscr{C}([0,T],L^2(D))\big)$
and $ \bar{\chi}$ in $L^{2}(\Omega\times Q)$   such that
\begin{equation*}
\begin{array}{ll}
(i)&\tilde{\chi}^{\h},  \chi^{\h} \rightharpoonup \chi \mbox{ in } L^2((0,T)\times\Omega,H^1(D)), \\
(ii)&\nabla \tilde{\chi}^{\h}, \nabla \chi^{\h} \wstarconverge \nabla \chi \mbox{ in } L^\infty(0,T;L^2(\Omega\times D)),\\
(iii)&\di\tilde{\chi}^{\h}-\tilde{B}^{\h}\rightharpoonup \chi-\int_{0}^{.}hdw \text{ in } L^{2}(\Omega,H^1(Q)), \\
(iv)&\ae\big(\partial_t(\tilde{\chi}^{\h}-\tilde{B}^{\h})\big) \rightharpoonup \bar{\chi} \mbox{ in } L^2(\Omega\times Q), \\
(v)&\big(\tilde{\chi}^{\h}-\tilde{B}^{\h}\big)(0) \di\rightharpoonup \chi(0)\text{ in }L^{2}(\Omega\times D).\\
\end{array}
\end{equation*}
\end{prop}
\begin{proof} $(i)$ Thanks to (\ref{utmu})-(\ref{utmubis})-(\ref{butuut}) and (\ref{butuutbis}), there exists $\chi$  in $L^2((0,T)\times\Omega,H^1(D))$ such that, up to subsequences denoted in the same way, we have
\begin{equation*}
 \tilde{\chi}^{\h},\chi^{\h}, \tilde{\chi}^{\h}(.-\h) \rightharpoonup \chi\text{  in }L^2((0,T)\times\Omega,H^1(D)).
\end{equation*}
Since $\tilde{\chi}^{\h}(.-\h)$ belongs to the Hilbert space $\mathcal{N}^2_w(0,T,H^{1}(D))$ endowed with the norm of $L^2((0,T)\times \Omega,H^1(D))$, one gets that $\chi$ is also in  $\mathcal{N}^2_w(0,T,H^{1}(D))$.\\
$(ii)$ Using (\ref{2c1})-(\ref{utmu}), one gets directly that up to subsequences denoted in the same way,
\begin{equation*}
\nabla \tilde{\chi}^{\h},\nabla \chi^{\h} \wstarconverge \nabla \chi\text{  in }L^\infty(0,T;L^2(\Omega\times D)).
\end{equation*}
$(iii)$ Thanks to (\ref{1c1})-(\ref{6c1})-(\ref{5c1}), there exists $\zeta$ in  $L^\infty(0,T;L^2(\Omega\times D))$ and $L^{2}\big(\Omega, H^1(Q)\big)$ such that, up to a subsequence, $$\tilde{\chi}^{\h}-\tilde{B}^{\h}\rightharpoonup \zeta \text{  in }L^{2}\big(\Omega, H^1(Q)\big)\text{ and }\tilde{\chi}^{\h}-\tilde{B}^{\h} \wstarconverge \zeta \text{  in } L^\infty(0,T;L^2(\Omega\times D)).$$ Using Proposition \ref{CVBc1}, one gets by uniqueness of the limit that $$\di\zeta=\chi-\int_{0}^{.}hdw.$$
$(iv)$ Due to the Lipschitz property of $\ae$ and (\ref{1c1}), $\ae\big(\partial_t(\tilde{\chi}^{\h}-\tilde{B}^{\h})\big) $ is bounded in $L^{2}(\Omega\times Q)$ and there exists $\bar{\chi}$ in the same space such that, up to a subsequence
\begin{equation*}
\ae\big(\partial_t(\tilde{\chi}^{\h}-\tilde{B}^{\h})\big) \rightharpoonup \bar{\chi} \mbox{ in } L^2(\Omega\times Q).
\end{equation*} \\
$(v)$ Since $L^{2}\big(\Omega, H^1(Q)\big)$ is continuously embedded in $L^2\big(\Omega,\mathscr{C}([0,T], L^{2}(D))\big)$, one gets that $\di \chi-\int_{0}^{.}hdw$ belongs to $L^2\big(\Omega,\mathscr{C}([0,T], L^{2}(D))\big)$. Moreover, as the It\^o integral of an $\mathcal{N}_{w}^{2}(0,T,L^{2}(D))$ process is a continuous square integrable $L^{2}(D)$-valued martingale (see \cite{DaPratoZabczyk}), $\di\int_{0}^{.}hdw$ is in $L^{2}\left(\Omega, \mathscr{C}([0,T], L^{2}(D))\right)$. Thus $\chi$ belongs to $L^2\big(\Omega,\mathscr{C}([0,T], L^{2}(D))\big)$ and finally $\chi$ is an element of $\mathscr{C}([0,T],L^{2}(\Omega\times D))$. Particularly, we have $$\tilde{\chi}^{\h}(0)-\tilde{B}^{\h}(0)\rightharpoonup \di \big(\chi-\int_{0}^{.}hdw\big)(0)=\chi(0)\text{ in }L^{2}(\Omega\times D).$$
\end{proof}
Using these convergence results, let us derive some properties satisfied by the weak limits $\vartheta$ and $\chi$.

\subsection{Properties of the weak limits $\vartheta$ and $\chi$}

\begin{prop}\label{initc}$\vartheta(0)=\vartheta_0$ and $\chi(0)=\chi_0$ in $L^2(D)$.
\end{prop}

\begin{proof} Thanks to Proposition \ref{CVuc1} and Proposition \ref{CVuc1bis}, we have
$$\tilde{\vartheta}^{\h}(0) \di\rightharpoonup \vartheta(0)\text{ and }\big(\tilde{\chi}^{\h}-\tilde{B}^{\h}\big)(0) \di\rightharpoonup \chi(0)\text{ in }L^{2}(\Omega\times D).$$
Note that $\tilde{\vartheta}^{\h}(0)=\vartheta_0$ and $\big(\tilde{\chi}^{\h}-\tilde{B}^{\h}\big)(0)=\chi_0$. Using the fact that $\chi_0$ and $\vartheta_0$ are deterministic, one concludes that the announced result holds in $L^2(D)$.
\end{proof}

\begin{prop}\label{CoGr1} The following results hold:
\begin{enumerate}
\item[($i$)] The application $t\in[0,T]\mapsto E\l[\|\nabla \vartheta(t)\|^{2}\r]\in \R$ is continuous.
\item[($ii$)] $\vartheta$ belongs to the space $\mathscr{C}\l([0,T], L^{2}(\Omega, H^1(D))\r)$.
\end{enumerate}
\end{prop}
\begin{proof}($i$) Note that $P$-almost surely in $\Omega$, $\di \vartheta$ satisfies the heat equation
\begin{equation*}
\left\{\begin{array}{rlc}
\partial_{t}\vartheta-\Delta \vartheta&=&-\partial_t U, \label{eqchal1c11}\\
\vartheta(0,.)&=&\vartheta_{0},
\end{array}\right.
\end{equation*}
where $\di U=\chi-\int_{0}^{.}hdw$.  Since $\vartheta_{0}\in H^1(D)$, the study of the heat equation gives us the following energy equality (see \cite{Brezisbis} Theorem X.11 p.220), for any $t\in[0,T]$  by denoting $Q_{t}=(0,t)\times D$:
\begin{align*}
&\int_{ Q_{t}}|\partial_{t}\vartheta|^2dsdx + \int_{ Q_{t}}\partial_{t}U\partial_{t}\vartheta dsdx + \frac{1}{2}||\nabla \vartheta(t)||^2 = \  \frac{1}{2} ||\nabla \vartheta_{0}||^2.
\end{align*}
Then by taking the expectation:
\begin{align}\label{ContF1}
&E\left[\int_{ Q_{t}}|\partial_{t}\vartheta|^2dsdx \right]+E\left[\int_{ Q_{t}}\partial_{t}U\partial_{t}\vartheta dsdx\right] \nonumber\\
&+ \frac{1}{2}E\left[||\nabla \vartheta(t)||^2\right]= \  \frac{1}{2} E\left[||\nabla \vartheta_{0}||^2\right].
\end{align}
Using (\ref{ContF1}) and the properties of the Lebesgue integral, one gets the continuity of $$t \in [0,T]\mapsto E\left[\|\nabla \vartheta(t)\|^{2}\right]\in\R.$$
($ii$) Firstly, since $\vartheta$ belongs to $L^2(\Omega, H^1(Q))$, it is also an element of \\$\mathscr{C}\l([0,T], L^{2}(\Omega, L^2(D))\r)$. Combining this with the continuity result proved  in ($i$), one gets that the application 
$$t \in [0,T]\mapsto E\left[\| \vartheta(t)\|^{2}_{H^1(D)}\right]\in\R$$
is also continuous. Secondly, that thanks to the following embedding (see \cite{Lions} Lemme 8.1 p.297):
$$L^{\infty}(0,T; L^2(\Omega, H^1(D))\cap \mathscr{C}\hspace*{-0.05cm}\l([0,T], L^{2}(\Omega, L^2(D))\r)\hspace*{-0.07cm}\subset\hspace*{-0.07cm} \mathscr{C}_w\hspace*{-0.05cm}\l([0,T], L^{2}(\Omega, H^1(D))\r)\footnotemark[2]$$
\footnotetext[2]{$\mathscr{C}_w\l([0,T], L^{2}(\Omega, H^1(D))\r)$ denotes the set of functions defined on $[0, T ]$ with values in $L^{2}(\Omega, H^1(D))$ which are weakly continuous.}one shows that $\vartheta$ also belongs to $\mathscr{C}\l([0,T], L^{2}(\Omega, H^1(D))\r)$.\\ Thus, combining this with the above continuity result, one concludes that $\vartheta$ is an element of $\mathscr{C}\l([0,T], L^{2}(\Omega, H^1(D))\r)$.
\end{proof}

\section{Proof of Theorem \ref{EUc1}} \label{section5c1}
\noindent Thanks to the weak convergence results stated in the previous section, passing to the limit in (\ref{dc1})-(\ref{dc1bis}) with respect to $\h$  is now possible and gives, using the separability of $H^1(D)$, $t$-almost everywhere in $(0,T)$, $P$-almost surely in $\Omega$ and for any $v$ in $H^1(D)$
\begin{align*}
\int_{D}\partial_t\vartheta vdx+
\int_D\partial_t\big(\chi-\int_0^.hdw\big)vdx+\int_D\nabla \vartheta.\nabla vdx=0\\
\int_D\partial_t\big(\chi-\int_0^.hdw\big)vdx+\int_{D}\bar{\chi}vdx+\int_D\nabla \chi.\nabla vdx=\int_D\vartheta vdx.
\end{align*}
Then it remains to identify the nonlinear weak limit $\bar{\chi}$ in $L^{2}(\Omega\times Q)$ of $\alpha(\partial_t(\tilde{\chi}^{\h}-\tilde{B}^{\h}))$. To do so, we suppose  in a first step (only for technical reasons) that $h$ belongs to $\di\mathcal{N}^2_w(0,T,H^{2}(D))$ by following the idea of \cite{BLM}. In a second step (Subsection \ref{step2}), we will obtain the well-posedness result for $h$ in $\di\mathcal{N}^2_w(0,T,H^{1}(D))$.\\

\subsection{Existence result for (\ref{eqc1}) when $h\in\mathcal{N}^{2}_{w}(0,T,H^2(D))$}

\begin{prop} \label{energiec1}Assume that $h$ belongs to $\di\mathcal{N}^2_w(0,T,H^{2}(D))$. Then, up to a subsequence,
\begin{align*}
\ae\l(\partial_{t}(\tilde{\chi}^{\h}-\tilde{B}^{\h})\r)\rightharpoonup \ae\big(\partial_{t}(\chi-\int_{0}^{.}hdw)\big) \text{ in } L^{2}(\Omega\times Q).
\end{align*}
\end{prop}
\begin{proof} 
\noindent Set $n$ in $\{0,...,N-1\}$. We introduce for the sequel the notations $$\di U=\chi-\int_{0}^{.}hdw\text{ and }U_{n+1}=\chi_{n+1}-\sum_{k=0}^{n}(w_{k+1}-w_{k})h_{k}.$$
Firstly, we consider  (\ref{bc2}) with  the test function 
$$\di v=\frac{U_{n+1}-U_{n}}{\h}=\frac{\chi_{n+1}-\chi_{n}}{\h} - h_{n}\frac{w_{n+1}-w_{n}}{\h}.$$
Thus, one gets $P$-a.s in $\Omega$:
\begin{align}\label{vte}
&\int_{D} \left(\frac{U_{n+1}-U_{n}}{\h}\right)^2dx +\int_{D} \ae\left(\frac{U_{n+1}-U_{n}}{\h}\right) \l(\frac{U_{n+1}-U_{n}}{\h}\r)dx\nonumber\\
& +\int_D\nabla U_{n+1}.\nabla\l(\frac{U_{n+1}-U_{n}}{\h}\r)dx\\
=&  \ \sum_{k=0}^{n}(w_{k+1}-w_{k})\int_{D}\Delta h_{k} \l(\frac{U_{n+1}-U_{n}}{\h}\r) dx+\int_D \l(\frac{U_{n+1}-U_{n}}{\h}\r) \vartheta_{n+1}dx.\nonumber
\end{align}
Secondly, (\ref{bc1}) with the test function $v=\vartheta_{n+1}$ gives $P$-a.s in $\Omega$:
\begin{align*}
\di\int_{D}\Big(\frac{\vartheta_{n+1}-\vartheta_n}{\Delta t}\Big)\vartheta_{n+1}dx+\int_{D}\left(\frac{U_{n+1}-U_{n}}{\h}\right)\vartheta_{n+1}dx+\int_D|\nabla \vartheta_{n+1}|^2dx=0
\end{align*}
and then
\begin{align*}
&\int_{D}\left(\frac{U_{n+1}-U_{n}}{\h}\right)\vartheta_{n+1}dx\\
=&-\frac{1}{2\h}\Big[ || \vartheta_{n+1}||^2-|| \vartheta_{n}||^2+|| \vartheta_{n+1}- \vartheta_{n}||^2\Big]-||\nabla \vartheta_{n+1}||^2.
\end{align*}
Injecting this in  (\ref{vte}), we obtain
\begin{align*}
&\h\int_{D} \left(\frac{U_{n+1}-U_{n}}{\h}\right)^2dx +\h\int_{D} \ae\left(\frac{U_{n+1}-U_{n}}{\h}\right)\l(\frac{U_{n+1}-U_{n}}{\h}\r)dx \\
&+ \frac{1}{2}\big(||\nabla U_{n+1}||^2 -  ||\nabla U_{n}||^2\big)+ \frac{1}{2}\big(|| \vartheta_{n+1}||^2-|| \vartheta_{n}||^2\big)\\
&+\frac{1}{2}|| \vartheta_{n+1}- \vartheta_{n}||^2+\h||\nabla \vartheta_{n+1}||^2\\
 \leq&\ \h\int_{D}\Delta B_{n+1} \frac{U_{n+1}-U_{n}}{\h} dx.
\end{align*}
By adding from $n=0$ to $N-1$, we get
\begin{align*}
&\h\sum_{n=0}^{N-1}\int_{D} \left(\frac{U_{n+1}-U_{n}}{\h}\right)^2dx+\h\sum_{n=0}^{N-1}\int_{D} \ae\left(\frac{U_{n+1}-U_{n}}{\h}\right)\l(\frac{U_{n+1}-U_{n}}{\h}\r)dx \\
&+ \frac{1}{2}\sum_{n=0}^{N-1}\big(||\nabla U_{n+1}||^2 - ||\nabla U_{n}||^2\big)+ \frac{1}{2}\sum_{n=0}^{N-1}\big(||\vartheta_{n+1}||^2 - ||\vartheta_{n}||^2\big) +\h\sum_{n=0}^{N-1}||\nabla \vartheta_{n+1}||^2\\
\leq&\
\h\sum_{n=0}^{N-1}\int_{D}\Delta B_{n+1} \frac{U_{n+1}-U_{n}}{\h} dx,
\end{align*}
and then
\begin{align*}
&\displaystyle\int_Q\big|\partial_t(\tilde{\chi}^{\h}-\tilde{B}^{\h})\big|^2dtdx+ \int_{Q} \ae\big(\partial_t(\tilde{\chi}^{\h}-\tilde{B}^{\h})\big)\partial_t(\tilde{\chi}^{\h}-\tilde{B}^{\h})dtdx \\
&+ \frac{1}{2}\big(\|\nabla U_{N}\|^{2}-\|\nabla U_{0}\|^{2}\big)+ \frac{1}{2}\big(\| \vartheta_{N}\|^{2}-\|\vartheta_{0}\|^{2}\big)+\int_Q|\nabla\vartheta^{\h}|^2dtdx
\\ \leq&
\displaystyle\int_{Q}\Delta B^{\Delta t}\partial_{t}(\tilde{\chi}^{\h}-\tilde{B}^{\h})dtdx.
\end{align*}

Noticing that $\di\nabla \tilde{U}^{\Delta t}(T)=\nabla U_{N}$ and that $\tilde\vartheta^{\h}(T)=\vartheta_N$, we finally get after taking the expectation
\begin{align}
&E\l[\int_Q\big|\partial_t(\tilde{\chi}^{\h}-\tilde{B}^{\h})\big|^2dtdx\r]+ \frac{1}{2}E\l[||\nabla \tilde{U}^{\h}(T)||^2\r]\nonumber\\
&+ \frac{1}{2}E\l[|| \tilde{\vartheta}^{\h}(T)||^2\r]+E\l[\int_Q|\nabla\vartheta^{\h}|^2dtdx\r]\nonumber\\
&+E\l[\int_{Q} \ae\left(\partial_t(\tilde{\chi}^{\h}-\tilde{B}^{\h})\right)\partial_t(\tilde{\chi}^{\h}-\tilde{B}^{\h})dtdx\r]\nonumber\\
 \leq&\ E\l[\int_{Q}\Delta B^{\h} \partial_t(\tilde{\chi}^{\h}-\tilde{B}^{\h}) dtdx\r]+\frac{1}{2} E\l[||\nabla U_0||^2 \r]+\frac{1}{2} E\l[||\vartheta_0||^2 \r].\label{carrec1} 
\end{align}

Now, passing to the superior limit in (\ref{carrec1}), we have using Proposition \ref{CVuc1} and Proposition \ref{CVuc1bis}
\begin{align*}
&\di\lim\inf_{\h\rightarrow 0}\|\partial_t(\tilde{\chi}^{\h}-\tilde{B}^{\h})\|^2_{L^2(\Omega\times Q)}+\di\frac{1}{2}\lim\inf_{\h\rightarrow 0}E\l[\|\nabla\tilde{ U}^{\h}(T)\|^{2}\r] \nonumber \\
&+\di\frac{1}{2}\lim\inf_{\h\rightarrow 0}E\l[\|\tilde{ \vartheta}^{\h}(T)\|^{2}\r] +\lim\inf_{\h\rightarrow 0}\|\nabla\vartheta^{\h}\|^{2}_{L^2(\Omega\times Q)}\nonumber \\
&+\lim\sup_{\h\rightarrow 0}E\di\l[\int_{Q}\ae(\partial_{t}(\tilde{\chi}^{\h}-\tilde{B}^{\h}))\partial_{t}(\tilde{\chi}^{\h}-\tilde{B}^{\h})dtdx\r]\nonumber\\
\leq&\ \di E\l[\int_{Q}\int_{0}^{t}\Delta hdw\partial_{t}Udtdx\r]+\frac{1}{2} E\l[||\nabla \chi_0||^2 \r]+\frac{1}{2} E\l[||\vartheta_0||^2 \r]. 
\end{align*}
Indeed, due to Remark \ref{rcvh}, $\di B^{\h}$ converges strongly in
 $L^{2}\l((0,T)\times \Omega, H^{2}(D)\r)$  to $\int_{0}^{.}hdw$ and so, by continuity of the Laplace operator, $\Delta B^{\h}$ converges strongly in $L^{2}(\Omega\times Q)$ to $\Delta \int_{0}^{.}hdw$. Moreover, following \cite{PrevotRockner} (Lemma 2.4.1 p.35),  $$\Delta \int_{0}^{.}hdw= \int_{0}^{.}\Delta hdw,$$
 and the convergence holds. Note that thanks to the embedding (see \cite{Lions} Lemme 8.1 p.297):
$$L^{\infty}(0,T; L^2(\Omega, H^1(D))\cap \mathscr{C}\hspace*{-0.05cm}\l([0,T], L^{2}(\Omega, L^2(D))\r)\hspace*{-0.07cm}\subset\hspace*{-0.07cm} \mathscr{C}_w\hspace*{-0.05cm}\l([0,T], L^{2}(\Omega, H^1(D))\r)\footnotemark[3]$$
\footnotetext[3]{$\mathscr{C}_w\l([0,T], L^{2}(\Omega, H^1(D))\r)$ denotes the set of functions defined on $[0, T ]$ with values in $L^{2}(\Omega, H^1(D))$ which are weakly continuous.}one gets that for all times $t$ in $[0,T]$, $\tilde{U}^{\h}(t)$ belongs to $L^{2}(\Omega, H^1(D))$ and $\tilde{U}^{\h}(t)\rightharpoonup U(t)$ in $L^{2}(\Omega, H^{1}(D))$. Then owing to the lower semi-continuity of the $L^{2}(\Omega, H^{1}(D))$-norm
\begin{eqnarray*}\liminf_{\h\rightarrow 0}E\l[\|\nabla\tilde{ U}^{\h}(T)\|^{2}\r]\geqslant E\l[||\nabla U(T)||^2\r]. 
\end{eqnarray*}
Using the same arguments on $\tilde\vartheta^{\h}$, one gets also that
\begin{eqnarray*}\liminf_{\h\rightarrow 0}E\l[\|\tilde{ \vartheta}^{\h}(T)\|^{2}\r]\geqslant E\l[||\vartheta(T)||^2\r]. 
\end{eqnarray*}
Finally, the lower semi-continuity of the $L^{2}(\Omega\times Q)$-norm gives us
\begin{align}
&\di||\partial _t(\chi-\int_0^.hdw)||^2_{L^2(\Omega\times Q)}+\frac{1}{2}E\l[||\nabla U(T)||^2\r]\nonumber\\
&+\frac{1}{2}E\l[\|\vartheta(T)\|^{2}\r]+||\nabla\vartheta||^2_{L^2(\Omega\times Q)}\nonumber\\
&+\limsup_{\h\rightarrow 0}E\di\l[\int_{Q}\ae(\partial_{t}(\tilde{\chi}^{\h}-\tilde{B}^{\h}))\partial_{t}(\tilde{\chi}^{\h}-\tilde{B}^{\h})dtdx\r]\nonumber\\
\leq&\ \di E\l[\int_{Q}\int_{0}^{t}\Delta hdw\partial_{t}Udtdx\r]+\frac{1}{2} ||\nabla \chi_0||^2 +\frac{1}{2} ||\vartheta_0||^2. \label{trianglec1}
\end{align}
Note that $P$-almost surely in $\Omega$, $\di U$ satisfies the heat equation
\begin{equation*}
\left\{\begin{array}{rll}
\partial_{t}U-\Delta U&=&g, \label{eqchal1c1}\\
U(0,.)&=&\chi_{0},
\end{array}\right.
\end{equation*}
where $\di g=\vartheta-\bar{\chi}+\int_{0}^{t}\Delta hdw$. Since $\chi_0\in H^1(D)$, the following energy equality holds  for any $t\in[0,T]$ (see \cite{Brezisbis} Theorem X.11 p.220):
\begin{align*}
&\int_{ Q_{t}}|\partial_{t}U|^2dsdx + \int_{ Q_{t}}\bar{\chi}\partial_{t}Udsdx + \frac{1}{2}||\nabla U(t)||^2 \\
=& \ \int_{Q_{t}}\int_0^s\Delta h dw \partial_tU dsdx +\int_{ Q_{t}}\vartheta\partial_{t}Udsdx+ \frac{1}{2} ||\nabla \chi_{0}||^2,
\end{align*}
where $Q_{t}=(0,t)\times D$.
Then, by taking the expectation:
\begin{align}\label{ContF}
&E\left[\int_{ Q_{t}}|\partial_{t}U|^2dsdx \right]+E\left[\int_{ Q_{t}}\bar{\chi}\partial_{t}Udsdx\right] + \frac{1}{2}E\left[||\nabla U(t)||^2\right]\\
 =& \ E\left[\int_{Q_{t}}\int_0^s\Delta h dw \partial_tU dsdx \right]+E\left[\int_{ Q_{t}}\vartheta\partial_{t}Udsdx\right]+ \frac{1}{2} ||\nabla \chi_{0}||^2.\nonumber
\end{align}
In the same manner, note that $\vartheta$ satisfies $P$-almost surely in $\Omega$ the  heat equation
\begin{equation*}
\left\{\begin{array}{rll}
\partial_t\vartheta-\Delta \vartheta&=&-\partial_tU, \\
\vartheta(0,.)&=&\vartheta_{0}\in H^1(D),
\end{array}\right.
\end{equation*}
and so the following energy equality holds for any $t\in [0,T]$ (see \cite{Temam95} Lemma 1.2 p260) :
\begin{align}\label{energyv}
&\frac{1}{2}E\l[||\vartheta(t)||^2\r]+||\nabla\vartheta||^2_{L^2(\Omega\times Q_t)}+E\l[\int_{Q_t}\vartheta\partial_tU dsdx\r]=\frac{1}{2}||\vartheta(0)||^2.
\end{align}
In this way, by injecting (\ref{ContF})-(\ref{energyv}) written with $t=T$ in  (\ref{trianglec1}) we finally have
\begin{align*}
&\lim\sup_{\h\rightarrow 0}E\l[\int_{Q} \ae\left(\partial_t(\tilde{\chi}^{\h}-\tilde{B}^{\h})\right)\partial_t(\tilde{\chi}^{\h}-\tilde{B}^{\h})dtdx\r] \\
\leq&\ E\l[\int_{Q}\bar{\chi}\partial_t\big(\chi-\int_{0}^{t}hdw\big)dtdx\r].
\end{align*}
As $\ae:\mathbb{R}\rightarrow \mathbb{R}$ is a Lipschitz continuous, non-decreasing function, the operator 
\begin{eqnarray*}
A_{\ae}: L^{2}(\Omega\times Q)&\rightarrow&  L^{2}(\Omega\times Q)\\
u&\mapsto&\ae(u),
\end{eqnarray*} 
is maximal monotone and one gets that $\di\bar{\chi}=\ae\Big(\partial_{t}(\chi-\int_{0}^{.}hdw)\Big)$ (see {\sc Lions} \cite{Lions} p.172).
\end{proof}

\begin{prop}\label{CoGr}Assume that $h$ belongs to $\di\mathcal{N}^2_w(0,T,H^{2}(D))$. Then, the following results hold
\begin{enumerate}
\item[($i$)] The application $t\in[0,T]\mapsto E\l[\|\nabla \chi(t)\|^{2}\r]\in \R$ is continuous.
\item[($ii$)] $\chi$ belongs to the space $\mathscr{C}\l([0,T], L^{2}(\Omega, H^1(D))\r)$.
\end{enumerate}
\end{prop}
\begin{proof} ($i$) Using (\ref{ContF}) and Lebesgue's theorem, one shows the continuity of $$t \in [0,T]\mapsto E\left[\|\nabla U(t)\|^{2}\right]\in\R,$$
where $\di U=\chi-\int_{0}^{.}hdw.$ Moreover, since $\di\int_{0}^{.}hdw$ is in $\mathscr{C}\l([0,T], L^{2}(\Omega, H^1(D))\r)$, one gets that the application $$t\in[0,T]\mapsto E\l[\|\int_0^t\nabla h(s)dw(s)\|^{2}\r]\in\R$$ is continuous and  the announced result holds.\\
($ii$) Due to (\ref{ContF}), Lebesgue's theorem and the fact that $\di U=\chi-\int_0^.hdw$ is an element of $\mathscr{C}\l([0,T], L^{2}(\Omega, L^2(D))\r)$, one gets the continuity of the application  $$t \in [0,T]\mapsto E\left[\| U(t)\|^{2}_{H^1(D)}\right]\in\R.$$
Note that as mentioned in the proof of Proposition \ref{energiec1}, $U$ belongs to the space $\mathscr{C}_w\l([0,T], L^{2}(\Omega, H^1(D))\r)$. Combining this with the above continuity result, one concludes that $U$ is in $\mathscr{C}\l([0,T], L^{2}(\Omega, H^1(D))\r)$ and due to the regularity of It\^o integral, it is the same for $\chi$.
\end{proof}

\begin{prop}\label{erhh2}
Assume that $h$ belongs to $\mathcal{N}^{2}_{w}(0,T,H^{2}(D))$. Then the couple $(\vartheta,\chi)$ given by Proposition \ref{CVuc1} and Proposition \ref{CVuc1bis} is a solution of System (\ref{eqc1}) in the sense of Definition \ref{ac1}.
\end{prop}

\begin{proof} Firstly, note that using Proposition \ref{CVuc1}, Proposition \ref{CVuc1bis}, Proposition \ref{CoGr1} and  Proposition \ref{CoGr}, $\vartheta$ and $\chi$ own regularities required by Definition \ref{ac1}. Secondly, they satisfy respectively the initial conditions $\vartheta(0,.)=\vartheta_0$ and $\chi(0,.)=\chi_{0}$ in $L^2(D)$ owing to Proposition \ref{initc}. Thirdly, thanks to Proposition \ref{energiec1}, by passing to the limit in (\ref{dc1})-(\ref{dc1bis}) with respect to $\h$ and  using the separability of $H^1(D)$, one gets $t$-almost everywhere in $(0,T)$, $P$-almost surely in $\Omega$ and for any $v$ in $H^1(D)$
\begin{equation*}
\left\{\begin{array}{rcl}
\di\int_{D}\partial_t\vartheta vdx+\int_D\partial_t\big(\chi-\int_0^.hdw\big)vdx+\int_D\nabla \vartheta.\nabla vdx=0&&\\
\di\int_D\tilde{\ae}\Big(\partial_t\big(\chi-\int_0^.hdw\big)\Big)vdx+\int_D\nabla \chi.\nabla vdx=\di\int_D\vartheta vdx.&&
\end{array}\right.
\end{equation*}
Hence, $(\vartheta,\chi)$ is a solution of  (\ref{eqc1}) in the sense of Definition \ref{ac1}
\end{proof}
 We now have all the necessary tools to show the result of existence and uniqueness for Problem (\ref{eqc1}) stated in Theorem \ref{EUc1}.
 
\subsection{Existence result for (\ref{eqc1}) when $h\in\mathcal{N}^{2}_{w}(0,T,H^1(D))$}\label{step2}

Assume that $h$ belongs to $\mathcal{N}^{2}_{w}(0,T,H^1(D))$. Owing to the density of $\mathscr{C}^{\infty}_c(\overline{D})$ in $H^1(D)$, we propose to approach $h$ by a sequence $(h_{n})_{n\in \N}\subset \mathcal{N}^{2}_{w}(0,T,\mathscr{C}^{\infty}_c(\overline{D})).$ Using Proposition \ref{erhh2}, one is able to define the following sequences :
\begin{defn} \label{solapp}Set $n,m\in\N$ and consider $\vartheta_{0}$, $\chi_{0}$ in $H^1(D)$  and $h_{n}, h_m$ belonging to $\mathcal{N}^{2}_{w}(0,T,\mathscr{C}^{\infty}_c(\overline{D}))$. We define the couples $(\vartheta_{n},\chi_{n})$ and $(\vartheta_{m},\chi_{m})$ as solutions of Problem (\ref{eqc1}) in the sense of Definition \ref{ac1} with the respective sets of data $(\vartheta_{0},\chi_{0},h_n)$ and $(\vartheta_{0},\chi_{0}, h_m)$. 
\end{defn}
For any $n,m\in \N$, we introduce the notations $$\di U_{n}=\chi_{n}-\int_{0}^{.}h_{n}dw\text{ and }\di U_{m}=\chi_{m}-\int_{0}^{.}h_{m}dw.$$ In what follows, our aim is to show  that $(\vartheta_n)_{n\in \N},(\chi_n)_{n\in \N}$ and $(U_n)_{n\in \N}$ are Cauchy sequences in suitable spaces.

\begin{lem} \label{lemcauchy}
The sequences $(\vartheta_n)_{n\in \N},(\chi_n)_{n\in \N}$ and $(U_n)_{n\in \N}$ introduced in Definition \ref{solapp} satisfy the following properties
\begin{enumerate}
\item[($i$)] $(\vartheta_n)_{n\in \N}$ is a  Cauchy sequence in $\mathcal{N}^{2}_{w}(0,T,H^{1}(D))\cap L^2(\Omega, H^1(Q))$.
\item[($ii$)] $(\chi_n)_{n\in \N}$ is a  Cauchy sequence in $\mathcal{N}^{2}_{w}(0,T,H^{1}(D))$.
\item[($iii$)] $(U_n)_{n\in \N}$ is a Cauchy sequence in $L^{2}\big(\Omega, H^1(Q)\big)$.
\item[($iv$)] For any $t$ in $[0,T]$, $(\vartheta_n(t))_{n\in \N}$ and $(\chi_n(t))_{n\in \N}$ are  Cauchy sequences in $L^2(\Omega, H^1(D))$.
\end{enumerate}
\end{lem}

\begin{proof}
Since the couples $(\vartheta_n,\chi_n)$ and $(\vartheta_m,\chi_m)$ satisfy  the variational formulation (\ref{dv1}) respectively with $h_n$ and $h_m$, one gets by subtracting them, $t$-almost everywhere in $(0,T)$, $P$-almost surely in $\Omega$ and for any $v$ in $H^1(D)$ that
\begin{align}\label{eqcauchy}
&\di\int_D\partial_t(\vartheta_n-\vartheta_m)vdx+\int_{D}\partial_{t}(U_{n}-U_{m})vdx\nonumber\\
&+\int_D\nabla (\vartheta_{n}-\vartheta_{m}).\nabla vdx=0.
\end{align}
For a fixed $t$ in $[0,T]$, we consider in (\ref{eqcauchy}) the test function $$\di v=\frac{(\vartheta_{n}-\vartheta_{m})(t)-(\vartheta_{n}-\vartheta_{m})(t-\h)}{\h}.$$
By noticing that 
\begin{align*}
&\di \int_D\nabla(\vartheta_{n}(t)- \vartheta_{m}(t)).\nabla\Big( \vartheta_{n}(t)- \vartheta_{m}(t) -\big( \vartheta_{n}(t-\h)- \vartheta_{m}(t-\h)\big)\Big)dx\smallskip\\
=&\ \di\frac{1}{2}\Big[\|\nabla(\vartheta_{n}-\vartheta_{m})(t)\|^{2}-\|\nabla(\vartheta_{n}-\vartheta_{m})(t-\h)\|^{2}\Big]\\
&+\frac{1}{2}\di\|\nabla(\vartheta_{n}-\vartheta_{m})(t)-\nabla(\vartheta_{n}-\vartheta_{m})(t-\h)\|^{2},
 \end{align*}
 we thus obtain 
\begin{align*}
&\di \int_D  \partial_t(\vartheta_n-\vartheta_m)\frac{(\vartheta_{n}- \vartheta_{m})(t)-(\vartheta_{n}- \vartheta_{m})(t-\h)}{\h}dx\smallskip\\
&+\di \int_D  \partial_t(U_n-U_m)\frac{(\vartheta_{n}- \vartheta_{m})(t)-(\vartheta_{n}- \vartheta_{m})(t-\h)}{\h}dx\smallskip\\
&\frac{1}{2\h}\Big[||\nabla (\vartheta_{n}- \vartheta_{m})(t)||^2 - ||\nabla (\vartheta_{n}-\vartheta_{m})(t-\h)||^2\Big] \\
+&\frac{1}{2\h}||\nabla (\vartheta_{n}- \vartheta_{m})(t)-\nabla(\vartheta_{n}- \vartheta_{m})(t-\h)||^2=0. 
\end{align*}
By taking the expectation and  the integral from $\h$ to $T$,  
one gets
\begin{align*}
&\di \int_{\h}^TE\left[\int_D  \partial_t(\vartheta_n-\vartheta_m)\frac{(\vartheta_{n}- \vartheta_{m})(t)-(\vartheta_{n}- \vartheta_{m})(t-\h)}{\h}dx\right]dt\\
+&\di \int_{\h}^TE\left[\int_D  \partial_t(U_n-U_m)\frac{(\vartheta_{n}- \vartheta_{m})(t)-(\vartheta_{n}- \vartheta_{m})(t-\h)}{\h}dx\right]dt\\
+& \frac{1}{2\h}\int_{\h}^T E\Big[||\nabla (\vartheta_{n}- \vartheta_{m})(t)||^2 - ||\nabla (\vartheta_{n}- \vartheta_{m})(t-\h)||^2\Big] dt
\\
+& \frac{1}{2\h}\int_{\h}^T E\big[||\nabla (\vartheta_{n}- \vartheta_{m})(t) - \nabla (\vartheta_{n}- \vartheta_{m})(t-\h)||^2\big] dt=0.
\end{align*}
Then, a change of variables gives us
\begin{align*}
&\di \int_{\h}^TE\left[\int_D  \partial_t(\vartheta_n-\vartheta_m)\frac{(\vartheta_{n}- \vartheta_{m})(t)-(\vartheta_{n}- \vartheta_{m})(t-\h)}{\h}dx\right]dt\\
+&\di \int_{\h}^TE\left[\int_D  \partial_t(U_n-U_m)\frac{(\vartheta_{n}- \vartheta_{m})(t)-(\vartheta_{n}- \vartheta_{m})(t-\h)}{\h}dx\right]dt\\
+& \frac{1}{2\h}\int_{T-\h}^T E\left[||\nabla (\vartheta_{n}- \vartheta_{m})(t)||^2\right]dt\leq \frac{1}{2\h}\int_{0}^{\h}E\left[||\nabla (\vartheta_{n}- \vartheta_{m})(t)||^2 \right]dt.
\end{align*}
By using Proposition \ref{CoGr1}, we obtain by  passing  to the limit with $\h$ and using the initial values:
\begin{align*}
&\di E\left[\int_Q | \partial_t(\vartheta_n-\vartheta_m)|^2dxdt\right]+ E\int_Q \partial_t(U_n-U_m)\partial_t(\vartheta_{n}- \vartheta_{m})dx\\
+&\frac{1}{2}E\left[||\nabla (\vartheta_{n}- \vartheta_{m})(T)||^2\right]\leq \frac{1}{2}E\left[||\nabla (\vartheta_{n}- \vartheta_{m})(0)||^2\right]=0.
\end{align*}
By denoting $Q_{t}=(0,t)\times D$, using the identity $ab=\frac{1}{2}[(a+b)^2-a^2-b^2]$ in the second term of the above equation and discarding nonnegative terms one  has (since $T$ is arbitrary) for any $t\in [0,T]$ 
\begin{align}\label{cde1}
&\di  || \partial_t(\vartheta_n-\vartheta_m)||^2_{L^2(\Omega\times Q_t)} +E\left[||\nabla (\vartheta_{n}- \vartheta_{m})(t)||^2\right]\\
\leq&\ ||\partial_t(U_n-U_m)||^2_{L^2(\Omega\times Q_t)}\nonumber.
 \end{align}
In the same manner, using the test function $\vartheta_n-\vartheta_m$ in (\ref{eqcauchy}), one shows the following inequality for any $t\in[0,T]$
\begin{align}\label{t2}
&E\left[||(\vartheta_n-\vartheta_m)(t)||^2\right]+2||\nabla (\vartheta_n-\vartheta_m)||^2_{L^2(\Omega\times Q_t)}\nonumber\\
\leq& \ ||\partial_t(U_n-U_m)||^2_{L^2(\Omega\times Q_t)}+||\vartheta_n-\vartheta_m||^2_{L^2(\Omega\times Q_t)}.
 \end{align} 
 Similarly, exploiting the fact that $(\vartheta_n,\chi_n)$ and $(\vartheta_m,\chi_m)$ satisfy  additionally the variational formulation (\ref{dv2}) respectively with $h_n$ and $h_m$, one gets by subtracting them, $t$-almost everywhere in $(0,T)$, $P$-almost surely in $\Omega$ and for any $v$ in $H^1(D)$ that
\begin{align}\label{eqcauchy1}
&\di\int_D\partial_t(U_n-U_m)vdx+\int_{D}\big(\ae(\partial_{t}(U_{n}))-\ae(\partial_{t}(U_{m}))\big)vdx\nonumber\\
&+\int_D\nabla (\chi_{n}-\chi_{m}).\nabla vdx=\int_D (\vartheta_n-\vartheta_m)vdx.
\end{align}
For a fixed $t$ in $[0,T]$, by taking  in (\ref{eqcauchy1}) the test function $$\di v=\frac{(U_{n}-U_{m})(t)-(U_{n}-U_{m})(t-\h)}{\h},$$ we get
\begin{align*}
&\di \int_D  \partial_t(U_n-U_m)\frac{(U_{n}- U_{m})(t)-(U_{n}- U_{m})(t-\h)}{\h}dx\\
&\di+\frac{1}{\h}\hspace*{-0.05cm}\int_D\hspace*{-0.05cm}\nabla(\chi_{n}(t)- \chi_{m}(t)).\nabla\Big( \chi_{n}(t)- \chi_{m}(t) - \big(\chi_{n}(t-\h)- \chi_{m}(t-\h)\big)\Big)dx\\
&-\di \frac{1}{\h}\int_D\nabla(\chi_{n}(t)-\chi_{m}(t)).\nabla\Big( \int_{t-\h}^{t} (h_{n}- h_{m})dw\Big)dx  \\
&+\di\int_{D}\big(\ae(\partial_t U_{n})-\ae(\partial_t U_{m})\big)\frac{(U_{n}- U_{m})(t)-(U_{n}- U_{m})(t-\h)}{\h}dx\\
&=\int_D (\vartheta_n-\vartheta_m)\frac{(U_{n}- U_{m})(t)-(U_{n}- U_{m})(t-\h)}{\h}dx.
\end{align*}
Then, by taking the expectation, the integral from $\h$ to $T$ and  using  changes of variables one arrives at (see \cite{BLM}, Proof of Theorem 1.4 for details)
\begin{align*}
&\di \int_{\h}^TE\left[\int_D  \partial_t(U_n-U_m)\frac{(U_{n}- U_{m})(t)-(U_{n}- U_{m})(t-\h)}{\h}dx\right]dt\\
+& \frac{1}{2\h}\int_{\h}^T E\left[||\nabla (\chi_{n}- \chi_{m})(t)||^2\right]dt - \frac{1}{2\h}\int_{0}^{\h}E\left[||\nabla (\chi_{n}- \chi_{m})(t)||^2 \right]dt\\
+&\int_{\h}^T \hspace*{-0.1cm}E\left[\int_{D}\big(\ae(\partial_tU_{n})-\ae(\partial_tU_{m})\big)\frac{(U_{n}-U_{m})(t)-(U_{n}- U_{m})(t-\h)}{\h}dx\right] dt
\\
\leq &\ ||\nabla(h_{n}- h_{m}) ||^2_{L^2(\Omega\times Q)}\\
&+\int_D (\vartheta_n-\vartheta_m)\frac{(U_{n}- U_{m})(t)-(U_{n}- U_{m})(t-\h)}{\h}dx.
\end{align*}
By passing to the limit with $\h$ in this inequality, using Proposition \ref{CoGr} and the initial value we obtain:
\begin{align*}
&\di E\left[\int_Q | \partial_t(U_n-U_m)|^2dxdt\right]+ 
\frac{1}{2}E\left[||\nabla (\chi_{n}- \chi_{m})(T)||^2\right]\\
+&E\left[\int_{Q}\big(\ae(\partial_tU_{n})-\ae(\partial_t U_{m})\big)\partial_t(U_{n}- U_{m})dxdt\right]\\
\leq &\ ||\nabla(h_{n}- h_{m}) ||^2_{L^2(\Omega\times Q)}+\int_D (\vartheta_n-\vartheta_m)\partial_t(U_{n}- U_{m})dx.
\end{align*}
Then, due to the  coercivity of $\ae$, one also has for any $t\in [0,T]$, by still denoting $Q_{t}=(0,t)\times D$,
\begin{align}\label{cde11}
&\big(\bar{C}_{\alpha} +\frac{1}{2}\big)||\partial_{t}(U_{n}-U_{m})||^{2}_{L^2(\Omega\times Q_t)}+\frac{1}{2}E\left[\|\nabla(\chi_{n}-\chi_{m})(t)\|^{2}\right]\\
\leq&\ ||\nabla(h_{n}- h_{m}) ||^2_{L^2(\Omega\times Q_t)}+\frac{1}{2}||\vartheta_n-\vartheta_m||^2_{L^2(\Omega\times Q_t)}. \nonumber
\end{align}
In the same manner, using the test function $U_n-U_m$ in (\ref{eqcauchy1}), one shows the following inequality for any $t\in[0,T]$
\begin{align}\label{t22}
&\frac{1}{2}E\left[||(U_n-U_m)(t)||^2\right]\hspace*{-0.05cm}+\hspace*{-0.05cm}\frac{1}{2}||\nabla (\chi_n-\chi_m)||^2_{L^2(\Omega\times Q_t)}\nonumber\\
\leq&\ C_{\alpha}^2 ||\partial_t(U_n-U_m)||^2_{L^2(\Omega\times Q_t)}+\frac{1}{2}||U_n-U_m||^2_{L^2(\Omega\times Q_t)}\nonumber\\
&+\frac{T}{2}||\nabla(h_n-h_m)||^2_{L^2(\Omega\times Q_t)}+||\vartheta_n-\vartheta_m||^2_{L^2(\Omega\times Q_t)}.
\end{align}
Now, by adding (\ref{t2}) and (\ref{t22}), one gets
\begin{align}\label{t222}
&E\left[||(\vartheta_n-\vartheta_m)(t)||^2\right]\hspace*{-0.05cm}
+\frac{1}{2}E\left[||(U_n-U_m)(t)||^2\right]\hspace*{-0.05cm}\nonumber\\
\leq& \ (C_{\alpha}^2+1) ||\partial_t(U_n-U_m)||^2_{L^2(\Omega\times Q_t)}+\frac{T}{2}||\nabla(h_n-h_m)||^2_{L^2(\Omega\times Q_t)}\nonumber\\
&+\frac{1}{2}||U_n-U_m||^2_{L^2(\Omega\times Q_t)}
+2||\vartheta_n-\vartheta_m||^2_{L^2(\Omega\times Q_t)}.
\end{align} 
From (\ref{cde11})
  \begin{align*}
 ||\partial_t(U_n-U_m)||^2_{L^2(\Omega\times Q_t)}\leq&\ \frac{1}{\bar{C}_{\alpha}\hspace*{-0.05cm} +\frac{1}{2}}\bigg\{
||\nabla(h_{n}- h_{m}) ||^2_{L^2(\Omega\times Q_t)}
 \hspace*{-0.05cm} +\hspace*{-0.05cm}\frac{1}{2} ||\vartheta_n-\vartheta_m||^2_{L^2(\Omega\times Q_t)}\hspace*{-0.05cm} \bigg\},
 \end{align*}
substituting it in (\ref{t222}), we obtain
 \begin{align*}
&E\left[||(\vartheta_n-\vartheta_m)(t)||^2\right]\hspace*{-0.05cm}
+\ \frac{1}{2}E\left[||(U_n-U_m)(t)||^2\right]\hspace*{-0.05cm}\nonumber\\
\leq&\ 
\left(\frac{C_{\alpha}^2+1}{\bar{C}_{\alpha}\hspace*{-0.05cm} +\frac{1}{2}}+\frac{T}{2}\right)||\nabla(h_{n}- h_{m}) ||^2_{L^2(\Omega\times Q)}\\
&+\frac{1}{2}||U_n-U_m||^2_{L^2(\Omega\times Q_t)}+\left(\frac{C_{\alpha}^2+1}{2\bar{C}_{\alpha}\hspace*{-0.05cm} +1}+2\right)||\vartheta_n-\vartheta_m||^2_{L^2(\Omega\times Q_t)}.
\end{align*}
By denoting for any $t$ in $[0,T]$ 
\begin{align*}
y(t)=&\ E\left[||(\vartheta_n-\vartheta_m)(t)||^2\right]
+\ \frac{1}{2}E\left[||(U_n-U_m)(t)||^2\right],\\
K^{n,m}_{\alpha}(t)=&
\left(\frac{C_{\alpha}^2+1}{\bar{C}_{\alpha} +\frac{1}{2}}+\frac{T}{2}\right)||\nabla(h_{n}- h_{m}) ||^2_{L^2(\Omega\times Q_t)}\hspace*{-0.07cm}\\
\text{ and } \hat{C_{\ae}}=&\ \frac{C_{\alpha}^2+1}{2\bar{C}_{\alpha} +1}+2,
 \end{align*}
 we have for any $t$ in $[0,T]$
 \begin{align}\label{hgl}
 y(t)\leq&\ K^{n,m}_{\alpha}(t)+\hat{C_{\ae}}\int_0^t y(s)ds\leq\ K^{n,m}_{\alpha}(T)+\hat{C_{\ae}}\int_0^t y(s)ds.
 \end{align}
Firstly, Gr\"onwall's Lemma then asserts that for any $t$ in $[0,T]$
  \begin{align}\label{gluv}
& \ E\left[||(\vartheta_n-\vartheta_m)(t)||^2\right]+\ \frac{1}{2}E\left[||(U_n-U_m)(t)||^2\right]\leq K^{n,m}_{\alpha}(T)e^{\hat{C_{\ae}} t}. 
 \end{align}
Thus, by taking the supremum over $[0,T]$ in (\ref{gluv}), one obtains the estimate
 \begin{align}\label{gluvbis}
& \ \sup_{t\in[0,T]} E\left[||(\vartheta_n-\vartheta_m)(t)||^2\right]+ \sup_{t\in[0,T]} \frac{1}{2}E\left[||(U_n-U_m)(t)||^2\right]\nonumber\\\leq&\ \left(2\hat{C_{\ae}}+\frac{T}{2}\right)e^{\hat{C_{\ae}} T}||\nabla(h_{n}- h_{m}) ||^2_{L^2(\Omega\times Q)}.
 \end{align}
Secondly, using (\ref{gluvbis})  in (\ref{cde11}) \textcolor{black}{allows us to affirm that}
  \begin{align}\label{edugc}
&\big(\bar{C}_{\alpha} +\frac{1}{2}\big)||\partial_{t}(U_{n}-U_{m})||^{2}_{L^2(\Omega\times Q)}+\frac{1}{2}\sup_{t\in [0,T]}E\left[\|\nabla(\chi_{n}-\chi_{m})(t)\|^{2}\right]\nonumber\\
\leq&\ \Big(\big(\hat{C_{\ae}}+\frac{T}{4}\big)Te^{\hat{C_{\ae}} T}+1\Big)||\nabla(h_{n}- h_{m}) ||^2_{L^2(\Omega\times Q)}. 
 \end{align}
 Thirdly, thanks to (\ref{edugc})  in (\ref{cde1}), we obtain
  \begin{align}\label{edvgv}
&\di \ || \partial_t(\vartheta_n-\vartheta_m)||^2_{L^2(\Omega\times Q)} +\sup_{t\in [0,T]}E\left[||\nabla (\vartheta_{n}- \vartheta_{m})(t)||^2\right]\nonumber\\
\leq&\ \ \frac{1}{\bar{C}_{\alpha} +\frac{1}{2}}\Big(\big(\hat{C_{\ae}}+\frac{T}{4}\big)Te^{\hat{C_{\ae}} T}+1\Big)||\nabla(h_{n}- h_{m}) ||^2_{L^2(\Omega\times Q)}. 
 \end{align}
 Finally, since $(h_{n})_{n}$ is a Cauchy sequence in $\mathcal{N}^{2}_{w}(0,T,H^{1}(D))$ and owing to (\ref{gluvbis}), (\ref{edugc}) and (\ref{edvgv}), one concludes that 
\begin{itemize}
\item[($i$)] $(\vartheta_n)_{n\in \N}$ is a  Cauchy sequence in $\mathcal{N}^{2}_{w}(0,T,H^{1}(D))\cap L^2(\Omega, H^1(Q))$.
\item[($ii$)] $(\chi_n)_{n\in \N}$ is a  Cauchy sequence in $\mathcal{N}^{2}_{w}(0,T,H^{1}(D))$.
\item[($iii$)] $(U_n)_{n\in \N}$ is a Cauchy sequence in $L^{2}\big(\Omega, H^1(Q)\big)$.
\item[($iv$)] For any $t$ in $[0,T]$, $(\vartheta_n(t))_{n\in \N}$ and $(\chi_n(t))_{n\in \N}$ are  Cauchy sequences in $L^2(\Omega, H^1(D))$.
\item[($v.$)]$(\partial_t\vartheta_n)_{n\in \N}$ is a Cauchy sequence in $L^2(\Omega\times Q)$.
\end{itemize}
\end{proof}
As mentioned by {\sc Da Prato-Zabczyk} \cite{DaPratoZabczyk}, $\mathcal{N}^{2}_{w}(0,T,H^{1}(D))$ is complete, thus due to Lemma \ref{lemcauchy}, the following convergence results hold directly.
\begin{cor}\label{191011_cor1}
There exist $\vartheta$ in $\mathcal{N}^{2}_{w}(0,T,H^{1}(D))\cap L^{2}\big(\Omega, H^1(Q)\big)$ and $\chi$ in\\ $\mathcal{N}^{2}_{w}(0,T,H^{1}(D))$ such that the sequences $(\vartheta_n)_{n\in \N},(\chi_n)_{n\in \N}$ and $(U_n)_{n\in \N}$ (introduced in Definition \ref{solapp}) satisfy the following convergence results
 \begin{equation*}
 \begin{array}{ccl}
 \di \vartheta_{n}\rightarrow \vartheta&\text{ in }&\mathcal{N}^{2}_{w}(0,T,H^{1}(D)) \text{ and }L^2(\Omega;H^1(Q)), \\ 
\di \chi_{n}\rightarrow \chi&\text{ in }&\mathcal{N}^{2}_{w}(0,T,H^{1}(D)), \\ 
 \di \partial_t \vartheta_{n}\rightarrow \partial_t\vartheta&\text{ in }&L^{2}(\Omega\times Q)\\
\di \partial_t U_{n}\rightarrow \partial_t\big(\chi-\int_{0}^{.}hdw\big)&\text{ in }&L^{2}(\Omega\times Q)\\
\forall t\in [0,T], \ \vartheta_n(t,.)\rightarrow \vartheta(t,.)&\text{ in }&L^2(\Omega,H^1(D)),\\
\forall t\in [0,T], \ \chi_n(t,.)\rightarrow \chi(t,.)&\text{ in }&L^2(\Omega,H^1(D)).\\
 \end{array}
 \end{equation*}
 \end{cor}
Note that since $(\vartheta_n)_{n\in\mathbb{N}}$ converges in\\ $L^2(\Omega;H^1(Q))$, it also converges in $L^{2}\big(\Omega, \mathscr{C}([0,T], L^{2}(D))\big)$ and, using the regularity of the It\^o integral, the same is true for $(\chi_n)_{n\in\mathbb{N}}$.
Thus, using Corollary \ref{191011_cor1}, we get that $t$-almost everywhere in $(0,T)$, $P$-almost surely in $\Omega$ and for any $v$ in $H^{1}(D)$
 \begin{align*}
\di\int_{D}\partial_t\vartheta v dx+\int_{D}\partial_t(\chi-\int_0^. h dw)vdx+\int_D\nabla \vartheta.\nabla vdx=0\\
\di\int_{D}\tilde{\ae}\Big( \partial_t(\chi-\int_0^. h dw)\Big)vdx+\int_D\nabla \chi.\nabla vdx=\int_D\vartheta vdx,
\end{align*}
and we have the existence result for $h\in \mathcal{N}^{2}_{w}(0,T,H^{1}(D))$ as announced in Theorem \ref{EUc1}.

\subsection{Uniqueness result for (\ref{eqc1})}
\begin{thm}\label{191011_thm01}
For $h$ in $\mathcal{N}^{2}_{w}(0,T,H^{1}(D))$ and initial data $(\chi_0,\vartheta_0)\in H^1(D)^2$ the solution in the sense of Definition \ref{ac1} of System (\ref{eqc1}) is unique.
\end{thm}
\begin{proof}
We consider $h$ in $\mathcal{N}^{2}_{w}(0,T,H^{1}(D))$, and $(\chi,\vartheta)$, $(\hat{\chi},\hat{\vartheta})$  two solutions in the sense of Definition \ref{ac1}  of System (\ref{eqc1}).
Using the notations $\di U=\chi-\int_{0}^{.}hdw$ and $\di \hat{U}=\hat{\chi}-\int_{0}^{.}hdw$, one has
\begin{equation*}
\left\{\begin{array}{rcl}
\di \partial_t(\vartheta-\hat{\vartheta})+\partial_t(U-\hat{U})-\Delta(\vartheta-\hat{\vartheta})=0& \text{in}&(0,T)\times D\times\Omega,\\
\di \partial_t(U-\hat{U})-\Delta (U-\hat{U})+\ae(\partial_{t}U)-\ae(\partial_{t}\hat{U})=\vartheta-\hat{\vartheta} &\text{in}&(0,T)\times D\times\Omega, \\
\nabla (U-\hat{U}).\mathbf{ n}=\nabla(\vartheta-\hat{\vartheta}).\mathbf{ n}=0 &\text{on}& (0,T)\times\partial D\times \Omega,\\
(U-\hat{U})(0,.)=0&\text{and}&(\vartheta-\hat{\vartheta})(0,.)=0.
\end{array}\right.
\end{equation*}
Denoting $\xi=\vartheta-\hat{\vartheta}$ and $u=U-\hat{U}$, it follows that $(\xi,u)$ is the solution of the following system of heat equations
\begin{equation*}
\left\{\begin{array}{rcl}
\di \partial_t\xi-\Delta \xi=-\partial_tu\quad &\text{in}&(0,T)\times D\times\Omega,\\
\di\partial_{t}u-\Delta u=\ae(\partial_{t}\hat{U})-\ae(\partial_{t}U)+\xi \quad &\text{in}&(0,T)\times D\times\Omega,  \\
\nabla u.\mathbf{ n}=\nabla\xi.\mathbf{ n}=0 \quad &\text{on}& (0,T)\times\partial D\times \Omega,\\
u(0,.)=0&\text{and}&\xi(0,.)=0. 
\end{array}\right.
\end{equation*}
Following the same arguments as in (\ref{energyv}), note that $\xi$ satisfies the energy equality
\begin{align*}
&\frac{1}{2}E\l[\|\xi(t)\|^2\r]+E\l[\int_{Q_t}|\nabla \xi|^{2}dsdx\r]
=\frac{1}{2}\|\xi(0)\|^{2}-E\l[\int_{Q_t}\xi\partial_{t}u dsdx\r],
\end{align*}
for  any $t$ in $[0,T]$ where $Q_t=(0,t)\times D$. Then,
\begin{align}
&\ E\l[\|\xi(t)\|^2\r]+2\|\nabla \xi\|^{2}_{L^{2}(\Omega\times Q_t)}\leq\|\partial_{t}u\|^2_{L^{2}(\Omega\times Q_t)}+\|\xi\|^2_{L^{2}(\Omega\times Q_t)}.\label{e5}
\end{align}
As in (\ref{ContF}), one has for any $t$ in $[0,T]$
\begin{align*}
&E\l[\int_{Q_t}|\partial_t u|^2dsdx\r]+\frac{1}{2}E\l[\|\nabla u(t)\|^{2}_{L^{2}(D)}\r]
=\frac{1}{2}\|\nabla u(0)\|^{2}_{L^{2}(D)}\\
&+E\l[\int_{Q_t}\xi\partial_t udsdx\r]- E\l[\int_{Q_t}\left(\ae(\partial_{t}U)-\ae(\partial_{t}\hat{U})\right)\partial_{t}udsdx\r].
\end{align*}
Note that this allows us to affirm thanks to Lebesgue's theorem that the application $t\in[0,T]\mapsto E\left[\|\nabla u(t)\|^{2}\right]\in\R$ is continuous. \\
Due to the coercivity  property of $\ae$, one gets for any $t$ in $[0,T]$
\begin{align*}
(\bar{C}_{\alpha}+\frac{1}{2})||\partial_{t}u||^{2}_{L^{2}(\Omega\times Q_t)}+\frac{1}{2}E\l[\|\nabla u(t)\|^{2}_{L^{2}(D)}\r]
\leq \frac{1}{2}\|\xi\|^{2}_{L^{2}(\Omega\times Q_t)},
\end{align*}
and then
\begin{align}
||\partial_{t}u||^{2}_{L^{2}(\Omega\times Q_t)}
\leq \frac{1}{1+2\bar{C}_{\alpha}}\|\xi\|^{2}_{L^{2}(\Omega\times Q_t)}\label{e6}.
\end{align}
On the one hand, going back to (\ref{e5}), in virtue of (\ref{e6}), we have for any $t$ in $[0,T]$
\begin{align*}
&\ E\l[\|\xi(t)\|^2\r]\leq \Big(1+\frac{1}{1+2\bar{C}_{\alpha}}\Big)\int_0^tE\l[\|\xi(s)\|^2\r]ds,
\end{align*}
and Gr\"onwall's Lemma allows us to assert that $\xi=0$, thus $\vartheta=\tilde\vartheta$.\\
On the other hand, the study of the heat equation also provides the following estimate on $u$ for any $t$ in $[0,T]$:
\begin{align*}
&\frac{1}{2}E\left[||u(t)||^2\right]-\frac{1}{2}||u(0)||^2+||\nabla u||^2_{L^2(\Omega\times Q_t)}\\
\leq&  \ ||u||^2_{L^2(\Omega\times Q_t)}+\frac{C_{\alpha}^2}{2}||\partial_{t}u||^{2}_{L^{2}(\Omega\times Q_t)}+\frac{1}{2}\|\xi\|^{2}_{L^{2}(\Omega\times Q_t)},
\end{align*}
which gives using (\ref{e6}) and the fact that $\xi=0$
\begin{align*}
&\frac{1}{2}E\left[||u(t)||^2\right]\leq||u||^2_{L^2(\Omega\times Q_t)}.
\end{align*} 
According to Gr\"onwall's Lemma, $u=0$ which implies that $\chi=\tilde\chi$ and the uniqueness result holds for (\ref{eqc1}).
\end{proof} 
 
\section{Proof of Proposition \ref{contdep}}
\noindent The proof of Lemma \ref{lemcauchy} allows us to show directly the following continuous dependence result  on the sequences $(\vartheta_n)_{n\in\N}$, $(\chi_n)_{n\in\N}$ given by Definition \ref{solapp}  with respect to the sequence of  integrands $(h_n)_{n\in\N}$ in the stochastic noise.

\subsection{Preliminary result}
\begin{lem}\label{lemdep1} There exists a constant $C_{\alpha}^T>0$ which only depends on  $T$, $C_{\alpha}$ and  $\bar{C}_{\alpha}$ such that the sequences $(\vartheta_n)_{n\in \N},(\chi_n)_{n\in \N}$ and $(U_n)_{n\in \N}$ introduced in Definition \ref{solapp} satisfy the following inequality for any $t$ in $[0,T]$
 \begin{align}\label{cecs}
 &\di \ || \partial_t(\vartheta_n-\vartheta_m)||^2_{L^2(\Omega\times Q_t)} +\big(\bar{C}_{\alpha} +\frac{1}{2}\big)||\partial_{t}(U_{n}-U_{m})||^{2}_{L^2(\Omega\times Q_t)}\nonumber\\
& \ +E\left[||(\vartheta_n-\vartheta_m)(t)||^2\right]+E\left[||\nabla (\vartheta_{n}- \vartheta_{m})(t)||^2\right]\nonumber\\
&+\ \frac{1}{4}E\left[||(\chi_n-\chi_m)(t)||^2\right]+\frac{1}{4}E\left[\|\nabla(\chi_{n}-\chi_{m})(t)\|^{2}\right]\nonumber\\
\leq&\ {C}_{\alpha}^T\left( ||h_n-h_m||^2_{L^2(\Omega\times Q_t)}+||\nabla(h_{n}- h_{m}) ||^2_{L^2(\Omega\times Q_t)}\right),
 \end{align}
 where $Q_{t}=(0,t)\times D$.
 \end{lem}
\begin{proof}
Owing to (\ref{hgl}) and by using again Gr\"onwall's Lemma one gets for any $t$ in $[0,T]$
  \begin{align*}
& \ E\left[||(\vartheta_n-\vartheta_m)(t)||^2\right]+\ \frac{1}{2}E\left[||(U_n-U_m)(t)||^2\right]\\
\leq& \ K^{n,m}_{\alpha}(t)+\int_0^t  \hat{C_{\ae}} K^{n,m}_{\alpha}(s)e^{\hat{C_{\ae}}(t-s)}ds, 
 \end{align*}
where $$K^{n,m}_{\alpha}(t)=\left(\frac{C_{\alpha}^2+1}{\bar{C}_{\alpha} +\frac{1}{2}}+\frac{T}{2}\right)||\nabla(h_{n}- h_{m}) ||^2_{L^2(\Omega\times Q_t)}\text{ and }\hat{C_{\ae}}=\ \frac{C_{\alpha}^2+1}{2\bar{C}_{\alpha} +1}+2.$$
Thus we obtain 
  \begin{align*}
& \ E\left[||(\vartheta_n-\vartheta_m)(t)||^2\right]+\ \frac{1}{2}E\left[||(U_n-U_m)(t)||^2\right]\\
\leq& \big(1+ \hat{C_{\ae}}Te^{\hat{C_{\ae}}T}\big)\left(\frac{C_{\alpha}^2+1}{\bar{C}_{\alpha} +\frac{1}{2}}+\frac{T}{2}\right)||\nabla(h_{n}- h_{m}) ||^2_{L^2(\Omega\times Q_t)}, 
 \end{align*}
And using this in (\ref{cde11}) allows us to affirm for any $t$ in $[0,T]$ that
  \begin{align*}
&\big(\bar{C}_{\alpha} +\frac{1}{2}\big)||\partial_{t}(U_{n}-U_{m})||^{2}_{L^2(\Omega\times Q_t)}+\frac{1}{2}E\left[\|\nabla(\chi_{n}-\chi_{m})(t)\|^{2}\right]\nonumber\\
\leq&\ \l\{\frac{T}{2}\big(1+ \hat{C_{\ae}}Te^{\hat{C_{\ae}}T}\big)\left(\frac{C_{\alpha}^2+1}{\bar{C}_{\alpha} +\frac{1}{2}}+\frac{T}{2}\right)+1\r\}||\nabla(h_{n}- h_{m}) ||^2_{L^2(\Omega\times Q_t)}. 
 \end{align*}
Now by injecting this last inequality  in (\ref{cde1}), we obtain for any $t$ in $[0,T]$
  \begin{align*}
&\di \ || \partial_t(\vartheta_n-\vartheta_m)||^2_{L^2(\Omega\times Q_t)} +E\left[||\nabla (\vartheta_{n}- \vartheta_{m})(t)||^2\right]\nonumber\\
\leq&\ \ \frac{1}{\bar{C}_{\alpha} +\frac{1}{2}} \l\{\frac{T}{2}\big(1+ \hat{C_{\ae}}Te^{\hat{C_{\ae}}T}\big)\left(\frac{C_{\alpha}^2+1}{\bar{C}_{\alpha} +\frac{1}{2}}+\frac{T}{2}\right)+1\r\}||\nabla(h_{n}- h_{m}) ||^2_{L^2(\Omega\times Q_t)}. 
 \end{align*}
 Finally, by noticing that for any $t$ in $[0,T]$
 $$E\l[||(\chi_n-\chi_m)(t)||^2\r]\leq 2E\l[||(U_n-U_m)(t)||^2\r]+2||h_n-h_m||^2_{L^2(\Omega\times Q_t)}$$
 one gets
 \begin{align*}
& \ E\left[||(\vartheta_n-\vartheta_m)(t)||^2\right]+\ \frac{1}{4}E\left[||(\chi_n-\chi_m)(t)||^2\right]\\
&+\di \ || \partial_t(\vartheta_n-\vartheta_m)||^2_{L^2(\Omega\times Q_t)} +E\left[||\nabla (\vartheta_{n}- \vartheta_{m})(t)||^2\right]\nonumber\\
&+\big(\bar{C}_{\alpha} +\frac{1}{2}\big)||\partial_{t}(U_{n}-U_{m})||^{2}_{L^2(\Omega\times Q_t)}+\frac{1}{2}E\left[\|\nabla(\chi_{n}-\chi_{m})(t)\|^{2}\right]\nonumber\\
\leq&\ \ (\frac{1}{\bar{C}_{\alpha} +\frac{1}{2}}+1) \l\{\frac{T}{2}\big(1+ \hat{C_{\ae}}Te^{\hat{C_{\ae}}T}\big)\left(\frac{C_{\alpha}^2+1}{\bar{C}_{\alpha} +\frac{1}{2}}+\frac{T}{2}\right)+1\r\}||\nabla(h_{n}- h_{m}) ||^2_{L^2(\Omega\times Q_t)}\\
&+ \big(1+ \hat{C_{\ae}}Te^{\hat{C_{\ae}}T}\big)\left(\frac{C_{\alpha}^2+1}{\bar{C}_{\alpha} +\frac{1}{2}}+\frac{T}{2}\right)||\nabla(h_{n}- h_{m}) ||^2_{L^2(\Omega\times Q_t)}\\
&+\frac{1}{2}||h_n-h_m||^2_{L^2(\Omega\times Q_t)}
 \end{align*}
and thus the existence of a constant ${C}_{\alpha}^T$ which only depends on $T, C_{\alpha}, \bar{C}_{\alpha}$ and $\hat{C_{\ae}}$ such that
 \begin{align*}
 &\di \ || \partial_t(\vartheta_n-\vartheta_m)||^2_{L^2(\Omega\times Q_t)} +\big(\bar{C}_{\alpha} +\frac{1}{2}\big)||\partial_{t}(U_{n}-U_{m})||^{2}_{L^2(\Omega\times Q_t)}\\
& \ +E\left[||(\vartheta_n-\vartheta_m)(t)||^2\right]+E\left[||\nabla (\vartheta_{n}- \vartheta_{m})(t)||^2\right]\nonumber\\
&+\ \frac{1}{4}E\left[||(\chi_n-\chi_m)(t)||^2\right]+\frac{1}{4}E\left[\|\nabla(\chi_{n}-\chi_{m})(t)\|^{2}\right]\\
\leq&{C}_{\alpha}^T\left( ||h_n-h_m||^2_{L^2(\Omega\times Q_t)}+||\nabla(h_{n}- h_{m}) ||^2_{L^2(\Omega\times Q_t)}\right).
 \end{align*}
\end{proof}

\subsection{Proof of Proposition \ref{contdep}}
Using a Cauchy sequence argument as in Subsection \ref{step2} and recalling the uniqueness result of Theorem \ref{191011_thm01}, one gets by passing to the limit in the inequality (\ref{cecs}) above the stability result announced in Proposition \ref{contdep}. More precisely, for $h,\hat{h}\in \mathcal{N}^{2}_{w}(0,T,H^{1}(D))$ there exist $(h_n)_{n\in\mathbb{N}}$ and $(\hat{h}_n)_{n\in\mathbb{N}}$ in $\mathcal{N}^{2}_{w}(0,T,\mathcal{C}^{\infty}_c(\overline{D}))$ such that $h_n\rightarrow h$, $\hat{h}_n\rightarrow\hat{h}$ in $\mathcal{N}^{2}_{w}(0,T,H^{1}(D))$ for $n\rightarrow\infty$. We fix initial data $(\vartheta_0,\chi_0)\in H^1(D)^2$ and consider solutions $(\vartheta_n,\chi_n)$, $(\hat{\vartheta}_n,\hat{\chi}_n)$ with data $(h_n,\vartheta_0,\chi_0)$, $(\hat{h}_n,\vartheta_0,\chi_0)$ respectively. Plugging $(\vartheta_n,\chi_n)=(\vartheta_n,\chi_n)$, $(\vartheta_m,\chi_m)=(\hat{\vartheta}_n,\hat{\chi}_n)$ into \eqref{cecs} and using the convergence results from Corollary \ref{191011_cor1}, we can pass to the limit and obtain
 \begin{align*}
 &\di \ || \partial_t(\vartheta-\hat{\vartheta})||^2_{L^2(\Omega\times Q_t)} +\big(\bar{C}_{\alpha} +\frac{1}{2}\big)||\partial_{t}(U-\hat{U})||^{2}_{L^2(\Omega\times Q_t)}\nonumber\\
& \ +E\left[||(\vartheta-\hat{\vartheta})(t)||^2\right]+E\left[||\nabla (\vartheta-\hat{\vartheta})(t)||^2\right]\nonumber\\
&+\ \frac{1}{4}E\left[||(\chi-\hat{\chi})(t)||^2\right]+\frac{1}{4}E\left[\|\nabla(\chi-\hat{\chi})(t)\|^{2}\right]\nonumber\\
\leq&\ {C}_{\alpha}^T\left( ||h-\hat{h}||^2_{L^2(\Omega\times Q_t)}+||\nabla(h-\hat{h}) ||^2_{L^2(\Omega\times Q_t)}\right),
 \end{align*}
 where $(\vartheta,\chi)$, $(\hat{\vartheta},\hat{\chi})$ are the unique solutions with data $(h,\vartheta_0,\chi_0)$, $(\hat{h},\vartheta_0,\chi_0)$ respectively, $U:=\chi-\int_0^{\cdot} h \ dw$, $\hat{U}:=\hat{\chi}-\int_0^{\cdot} \hat{h} \ dw$.

\section{Proof of Theorem \ref{Non}}

Under Assumptions $H_2$ to $H_4$, we are interested in the following system with multiplicative noise:
\begin{equation*}\label{eqc12}
\left\{\begin{array}{rcl}
\di \partial_t\vartheta+\partial_t(\chi-\int_0^. \mathscr{H}(\chi)dw)-\Delta\vartheta&=&0 \text{ in }(0,T)\times D\times\Omega,\\
\di \tilde{\ae}\left(\partial_t (\chi-\int_0^. \mathscr{H}(\chi)dw)\right) -\Delta \chi &=&\vartheta \text{ in }(0,T)\times D\times\Omega, \\
\nabla \chi.\mathbf{n}=\nabla\vartheta.\mathbf{n}&=&0 \text{ on } (0,T)\times\partial D\times \Omega,\\
\chi(0,.)=\chi_{0}&\text{and}&\vartheta(0,.)=\vartheta_0.
\end{array}\right.
\end{equation*}
\noindent Using  Theorem \ref{EUc1}, we define the application 
\begin{eqnarray*}
f: \mathcal{N}^{2}_{w}(0,T,H^{1}(D))&\rightarrow& \mathcal{N}^{2}_{w}(0,T,H^{1}(D))\times \mathcal{N}^{2}_{w}(0,T,H^{1}(D)),\\
S&\mapsto& \di (\vartheta_S,\chi_{S}),
\end{eqnarray*}
\noindent where $(\vartheta_S,\chi_{S})$ is the solution  of the following system with additive noise
\begin{equation*}\label{eqc15}
\left\{\begin{array}{rcl}
\di \partial_t\vartheta_S+\partial_t(\chi_S-\int_0^. \mathscr{H}(S)dw)-\Delta\vartheta_S&=&0\text{ in }(0,T)\times D\times\Omega,\\
\di \tilde{\ae}\left(\partial_t (\chi_S-\int_0^. \mathscr{H}(S)dw)\right) -\Delta \chi_S&=&\vartheta_S \text{ in }(0,T)\times D\times\Omega, \\
\nabla \chi_S.\mathbf{n}=\nabla\vartheta.\mathbf{n}&=&0 \text{ on } (0,T)\times\partial D\times \Omega,\\
\chi_S(0,.)=\chi_{0}&\text{and}&\vartheta_S(0,.)=\vartheta_0,
\end{array}\right.
\end{equation*}
\noindent in the sense of Definition \ref{ac1} with $h=\mathscr{H}(S)$. Additionally, we consider the following projection application 
\begin{eqnarray*}
g: \mathcal{N}^{2}_{w}(0,T,H^{1}(D))\times \mathcal{N}^{2}_{w}(0,T,H^{1}(D))&\rightarrow& \mathcal{N}^{2}_{w}(0,T,H^{1}(D)),\\
(u,v)&\mapsto& \di v.
\end{eqnarray*}
Our aim is to show that the composition $g\circ f$ admits a unique fixed-point in $\mathcal{N}^{2}_{w}(0,T,H^{1}(D))$. The idea is to exploit the fact that, for any $a>0$, the following exponential weight in time norm
\begin{eqnarray*}
\mathcal{N}^{2}_{w}(0,T,H^{1}(D))&\rightarrow &\R_+\\
v&\mapsto& \int_0^T e^{-at}E\big[||v(t)||^2_{H^1(D)}\big]dt
\end{eqnarray*}
provides an equivalent norm  to the usual one on $\mathcal{N}^{2}_{w}(0,T,H^{1}(D))$.
\quad \\
Set $a>0$, consider  $S$ and  $\hat{S}$ in $\mathcal{N}^{2}_{w}(0,T,H^{1}(D))$ and define  $f(S)=(\vartheta_S,\chi_S)$ and $f(\hat S)=(\vartheta_{\hat S},\chi_{\hat S})$. According to \cite{PrevotRockner} Lemma 2.41 p.35, note that $\mathscr{H}(S)$ and  $\mathscr{H}(\hat{S})$  belong to $\mathcal{N}^{2}_{w}(0,T,H^{1}(D))$. Using Proposition \ref{contdep}, one gets for any $t$ in $[0,T]$
\begin{align}\label{191015_01}
& \ E\left[||(\vartheta_S-\vartheta_{\hat S})(t)||^2\right]+E\left[||\nabla (\vartheta_S- \vartheta_{\hat S})(t)||^2\right]\nonumber\\
&+\ \frac{1}{4}E\left[||(\chi_{S}-\chi_{\hat S})(t)||^2\right]+\frac{1}{4}E\left[\|\nabla(\chi_S-\chi_{\hat S})(t)\|^{2}\right]\nonumber\\
\leq&\ {C}_{\alpha}^T\left( ||\mathscr{H}(S)-\mathscr{H}(\hat S)||^2_{L^2(\Omega\times Q_t)}+||\nabla(\mathscr{H}(S)- \mathscr{H}(\hat S)) ||^2_{L^2(\Omega\times Q_t)}\right)\nonumber\\
\leq&\ {C}_{\alpha}^T\int_0^tE\left[ ||\mathscr{H}(S)-\mathscr{H}(\hat S)||^2_{H^1(D)}\right]ds.
\end{align} 
We fix $t>0$. Multiplying \eqref{191015_01} with $e^{-at}$ and integrating over $(0,T)$, we obtain
\begin{align*}
&\int_0^Te^{-at}E\left[||(\chi_{S}-\chi_{\hat S})(t)||^2\right]dt+\int_0^Te^{-at}E\left[\|\nabla(\chi_S-\chi_{\hat S})(t)\|^{2}\right]dt\\
\leq&\ 4{C}_{\alpha}^T\int_0^T e^{-at}\int_0^tE\left[||(\mathscr{H}(S)-\mathscr{H}(\hat S))(s)||_{H^1(D)}^2\right]dsdt.
\end{align*}
By using an integration by parts one gets
 \begin{align*}
&\int_0^Te^{-at}E\left[||(\chi_{S}-\chi_{\hat S})(t)||^2\right]dt+\int_0^Te^{-at}E\left[\|\nabla(\chi_S-\chi_{\hat S})(t)\|^{2}\right]dt\\
\leq&\ 4{C}_{\alpha}^TC^2_{\mathscr{H}}\int_0^Te^{-at}\int_0^tE\left[||(S-\hat S)(s)||^2_{H^1(D)}\right]dsdt\\
=&\ \di 4{C}_{\alpha}^TC^2_{\mathscr{H}}\times \frac{1}{a}\int_{0}^{T}e^{-a t}E\left[\| (S-\hat{S})(t)\|^{2}_{H^1(D)}\right]dt \\
&-4{C}_{\alpha}^TC^2_{\mathscr{H}}\times\Big(\frac{1}{a}e^{-a T}\int_{0}^{T}E\left[\| (S-\hat{S})(t)\|^{2}_{H^1(D)}\right]dt\Big) \\
\leq&\ \di  4{C}_{\alpha}^TC^2_{\mathscr{H}}\times \frac{1}{a}\int_{0}^{T}e^{-a t}E\left[\| (S-\hat{S})(t)\|^{2}_{H^1(D)}\right]dt.
 \end{align*} 
Finally 
\begin{align*} 
&\int_{0}^{T}e^{-a t}E\left[||(g\circ f)(S)-(g\circ f)(\hat S)||_{H^{1}(D)}^2\right] dt\\
\leq&\ \di4{C}_{\alpha}^TC^2_{\mathscr{H}}\times \frac{1}{a}\int_{0}^{T}e^{-a t}E\left[\| (S-\hat{S})(t)\|^{2}_{H^1(D)}\right]dt.
\end{align*}
Under the condition $a>4{C}_{\alpha}^TC^2_{\mathscr{H}}$, $g\circ f$ is a contractive mapping, it has a unique fixed-point  and the result holds.

\subsection*{Acknowledgments}
The authors wish to thank the joint program between the \textit{Higher Education Commission from Pakistan Ministry of Higher Education} and the \textit{French Ministry of Foreign and European Affairs} for the funding of A. Maitlo's PhD thesis.\\
The authors wish to thank the PPP Programme for Project-Related Personal Exchange (France-Germany) for financial support.

\end{document}